\documentclass[10pt]{amsart}

\usepackage{amsmath}
\usepackage{mathrsfs}
\usepackage{amssymb}
\usepackage{amsthm}
\usepackage{enumerate}
\usepackage[top=1.5in, bottom=1.5in, left=1.25in, right=1.25in]{geometry}
\usepackage[hidelinks, bookmarks=false]{hyperref}

\usepackage[style=trad-plain, hyperref=true]{biblatex}
\providecommand{\N}{\mathbb{N}}
\providecommand{\Z}{\mathbb{Z}}

\providecommand{\R}{\mathbb{R}}
\providecommand{\C}{\mathbb{C}}

\providecommand{\E}{\mathbb{E}}
\providecommand{\P}{\mathcal{P}}
\providecommand{\T}{\mathbb{T}}
\providecommand{\CC}{\mathscr{C}}
\providecommand{\D}{\mathbb{D}}
\providecommand{\M}{\mathcal{M}}
\providecommand{\leqsim}{\lesssim}
\providecommand{\geqsim}{\gtrsim}

\providecommand{\F}{\mathcal{F}}
\providecommand{\I}{\mathbb{I}}

\renewcommand{\P}{\mathcal{P}}
\renewcommand{\H}{\mathcal{H}}
\renewcommand{\:}{\colon}

\numberwithin{equation}{section}

\newtheorem{theorem}{Theorem}[section]
\newtheorem{proposition}[theorem]{Proposition}
\newtheorem{corollary}[theorem]{Corollary}
\newtheorem{lemma}[theorem]{Lemma}

\theoremstyle{definition}

\theoremstyle{remark}

\title{Pointwise Convergence of Ergodic Averages Along Hardy Field Sequences}
\author{Maximilian O'Keeffe}
\date{}

\begin{document}
\maketitle

\begin{abstract}
Let $(X,\mu)$ be an arbitrary measure space equipped with a family of pairwise commuting measure preserving transformations $T_1, \dotsc, T_m$. We prove that the ergodic averages \[ A_{N;X}^{P_1, \dotsc, P_m}f = \frac{1}{N} \sum_{n=1}^N T_1^{\lfloor P_1(n) \rfloor} \dotsm T_m^{\lfloor P_m(n) \rfloor} f \] converge pointwise $\mu$-almost everywhere as $N \to \infty$ for any $f \in L^p(X)$ with $p>1$, where $P_1, \dotsc, P_m$ are Hardy field functions which are ``non-polynomial" and have distinct growth rates. To establish pointwise convergence we will prove a long-variational inequality, which will in turn prove that a maximal inequality holds for our averages.

Additionally, by restricting the class of Hardy field functions to those with the same growth rate as $t^c$ for $c>0$ non-integer, we also prove full variational estimates. We are therefore able to provide quantitative bounds on the rate of convergence of exponential sums of the form \[ \frac{1}{N} \sum_{n=1}^N e(\xi_1 \lfloor n^{c_1} \rfloor + \dotsb + \lfloor n^{c_m} \rfloor) \] where $0<c_1<\dotsb<c_m$ are non-integer.
\end{abstract}

\tableofcontents

\section{Introduction}

\subsection{Pointwise Convergence of Ergodic Averages} Consider a measure space $(X,\mu)$ which we allow to be $\sigma$-finite. Suppose $T_1, \dotsc, T_m\:X \to X$ are invertible, bi-measurable, commuting (in the sense that $T_i \circ T_j = T_j \circ T_i$ for all $i,j \in \{1, \dotsc, m\}$) and measure preserving preserving, so that $\mu(T_i^{-1}A) = \mu(A)$ for all measurable sets $A$. From now on we will denote $T_i \circ T_j$ by $T_iT_j$.

The study of pointwise convergence of ergodic averages began with Birkhoff \cite{Birkhoff89} who proved that the averages \[ \frac{1}{N} \sum_{n=1}^N f(T^nx) \] converge almost everywhere for any $f \in L^1(X)$. Furstenberg's ergodic theoretic proof of Szemer\'edi's theorem \cite{furstenberg1977ergodic} led to the study of other, more general, ergodic averages such as those of the form \begin{equation} \label{eqn:ergodicaverages} \frac{1}{N} \sum_{n=1}^N f \left( T_1^{P_1(n)} \dotsm T_m^{P_m(n)}x \right) \end{equation} where $P_1,\dotsc,P_m\:\Z \to \Z$ are integer sequences. One example of a class of sequences is to take the $P_i$ to be polynomials with integer coefficients. In the case $m=1$ Bourgain established the pointwise convergence of (\ref{eqn:ergodicaverages}) for any $f \in L^p(X)$ with $p>1$ by proving an oscillation estimate \cite{bourgain1988maximal,bourgain1988pointwise,Bourgain89}. More recently, Ionescu, Magyar, Mirek, and Szarek \cite{ionescu2023polynomial} proved that (\ref{eqn:ergodicaverages}) converge almost everywhere for any $f \in L^p(X)$ with $p>1$ for arbitrary $m \in \N$, any polynomials $P_1,\dotsc,P_m \in \Z[n]$ each having zero constant term, and for any $T_1, \dotsc, T_m$ which do not even have to be commuting (but which generate a two-step nilpotent group).

Another direction one can take is to let $P_1,\dotsc,P_m$ be real-valued sequences and to consider the averages \begin{equation} \label{eqn:floorergodicaverages} A_{N;X}^{P_1,\dotsc,P_m}f(x) = \frac{1}{N} \sum_{n=1}^N f \left( T_1^{\lfloor P_1(n) \rfloor} \dotsm T_m^{\lfloor P_m(n) \rfloor} x \right) \end{equation} Here a common example is to take $P_1,\dotsc,P_m$ from a \emph{Hardy field}. See Subsection \ref{subsec:hardyfield} for a precise definition but informally these are smooth functions such that they, and all their derivatives, are eventually monotonic---one can informally think of these as functions that look like $t^c$ for some $c>0$. This line of investigation began with Wierdl \cite{wierdl1989almost}, who established the pointwise convergence of (\ref{eqn:floorergodicaverages}) in the case $m=1$ and $P(t)=t^c$ where $c>1$ is non-integer.

In this paper we will take $P_1,\dotsc,P_m$ to be Hardy field functions such that the collection satisfies various properties which we will describe in Subsection \ref{subsec:hardyfield}. We will denote the set of such collections by $\P$. We give an informal outline what it means for $\{P_1,\dotsc,P_m\}$ to be in $\P$. It is well known that averages (\ref{eqn:ergodicaverages}) and (\ref{eqn:floorergodicaverages}) may fail to converge almost everywhere if $P_1$, say, is a function such as an exponential (which grows faster than any polynomial) or a logarithm (which grows slower than any polynomial). We therefore require $P_1, \dotsc, P_m$ to have polynomial growth. We will also require that $P_1,\dotsc,P_m$ are non-polynomial, so that the growth rates are non-integer. Finally, we will impose that $P_1,\dotsc,P_m$ have distinct growth rates.

Two examples of collections that belong in $\P$ to consider are $\{ t^{c_1},\dotsc,t^{c_m}\}$, where $c_1,\dotsc,c_m$ are positive, distinct and non-integer and $\{t^{a_1}(\log t)^{b_1},\dotsc,t^{a_m} (\log t)^{b_m}\}$ where $a_1,\dotsc,a_m$ are positive and distinct (but may be integer) and $b_1,\dotsc,b_m$ are non-zero reals.

\subsection{Statement of Results}

Our main theorem can be stated as follows.

\begin{theorem} \label{thm:mainthm}
Suppose $\{P_1,\dotsc,P_m\} \in \P$. Then the averages $A_{N;X}^{P_1,\dotsc,P_m}f$ converge pointwise almost everywhere as $N \to \infty$ for all $f \in L^p(X)$ with $1<p<\infty$.
\end{theorem}

To prove Theorem \ref{thm:mainthm} we will use a method employed by Bourgain in \cite{bourgain1988pointwise} and by Ionescu, Magyar Mirek, Szarek \cite{ionescu2023polynomial}. Indeed, the key observation is that qualitative behaviour such as almost everywhere pointwise convergence follows from certain quantitative estimates. In this paper we will prove a so-called long-variational inequality. For a sequence $(a_N)_{N \in \N}$ of complex numbers and $0<r<\infty$ we define the $r$-variation norm of $(a_N)$ by \begin{equation} V^r(a_N)_{N \in \N} = \sup_{N \in \N} |a_N| + \sup_{J \in \N} \sup_{\substack{N_0 \leq \dotsb \leq N_J \\ N_j \in \N}} \left( \sum_{j=0}^{J-1} |a_{N_{j+1}}-a_{N_j}|^r \right)^\frac{1}{r}. \end{equation} Note we can replace the set $\N$ in this definition with any other set $\D$ (so that, for example, the last supremum is taken over $N_j \in \D$). If $r \geq 1$ and $\D$ is a finite set then $V^r$ is a genuine norm on the space of sequences $(a_N)_{N \in \D}$.

The $r$-variation norms are useful in the context of pointwise convergence for the following reason. It is easily seen that if $V^r(a_N)_{N \in \N}<\infty$ then $(a_N)$ is a Cauchy sequence and hence converges. It therefore suffices to prove that $V^r(A_{N;X}^{P_1,\dotsc,P_m}f(x))_{N \in \N}<\infty$ for almost every $x \in X$ for some $r \geq 1$. This can be achieved by proving that the $L^p(X)$ norm of the function \[ x \mapsto V^r(A_{N;X}^{P_1,\dotsc,P_m}f(x))_{N \in \N} \] is finite. Thus it would suffice to prove an estimate of the form \[ \|V^r(A_{N;X}^{P_1,\dotsc,P_m}f)_{N \in \N} \|_{L^p(X)} \leqsim_{P_1,\dotsc,P_m,m,p,r} \|f\|_{L^p(X)}. \] The additional benefit of a variational estimate of this form is that one automatically obtains a maximal inequality and one can deduce norm convergence. This remains true if one replaces $\N$ with a lacunary set $\D$ (so there exists $\lambda>1$ such that $\frac{N_2}{N_1}>\lambda$ for all $N_1,N_2 \in \D$ with $N_1<N_2$) so long as the implicit constants do not depend on $\D$. An inequality of the form \[ \|V^r(A_{N;X}^{P_1,\dotsc,P_m}f)_{N \in \D} \|_{L^p(X)} \leqsim_{P_1,\dotsc,P_m,m,p,r,\lambda} \|f\|_{L^p(X)} \] for all $\lambda$-lacunary sets $\D$ is known as a long-variational inequality. Thus Theorem \ref{thm:mainthm} actually follows from the following.

\begin{proposition} \label{prop:variationalineq}
Suppose $\{P_1,\dotsc,P_m\} \in \P$ and let $f \in L^p(X)$ with $p>1$. Then for any $r>2$ and $\lambda>1$ we have the long-variational inequality \begin{equation} \label{eqn:variationalineq} \|V^r(A_{N;X}^{P_1,\dotsc,P_m}f)_{N \in \D} \|_{L^p(X)} \leqsim_{P_1,\dotsc,P_m,m,p,r,\lambda} \|f\|_{L^p(X)} \end{equation} whenever $\D \subseteq \N$ is $\lambda$-lacunary.
\end{proposition}

By the Calder\'on transference principle \cite{calderon1968ergodic} it suffices to prove the inequality (\ref{eqn:variationalineq}) in the case $X=\Z^m$, $\mu = \mu_{\Z^m}$ is the counting measure, and $T_i x = x-e_i$ for all $x \in \Z^m$, where $e_i$ is the $i$\textsuperscript{th} basis vector of $\R^m$. In this case we will denote $A_{N;\Z^m}^{P_1,\dotsc,P_m}$ by $A_N^{P_1,\dotsc,P_m}$. Note we have \[ A_N^{P_1,\dotsc,P_m}f(x) = \frac{1}{N} \sum_{n=1}^N f(x-(\lfloor P_1(n) \rfloor,\dotsc,\lfloor P_m(n) \rfloor)) \] for all $x \in \Z^m$. We can additionally consider the so-called upper half operator \[ \Tilde{A}_N^{P_1,\dotsc,P_m}f(x) = \frac{1}{N} \sum_{n=1}^N 1_{n>\frac{N}{2}} f(x-(\lfloor P_1(n) \rfloor,\dotsc,\lfloor P_m(n) \rfloor)). \] Finally, if we can prove that (\ref{eqn:variationalineq}) holds for all $\lambda$-lacunary subsets of $\N$ which are finite (with the implied constants independent of $\D$ and the cardinality of $\D$) then one also obtains (\ref{eqn:variationalineq}). Thus Proposition \ref{prop:variationalineq} is reduced to proving the following theorem.

\begin{theorem} \label{thm:integervariation}
Suppose $\{P_1,\dotsc,P_m\} \in \P$ and let $f \in \ell^p(\Z^m)$ with $p>1$. Then for any $r>2$ and $\lambda>1$ we have the long-variational inequality \begin{equation} \label{eqn:integervariation} \| V^r(\Tilde{A}_N^{P_1,\dotsc,P_m}f)_{N \in \D} \|_{\ell^p(\Z^m)} \leqsim_{P_1,\dotsc,P_m,p,r,\lambda} \|f\|_{\ell^p(\Z^m)} \end{equation} for all finite $\lambda$-lacunary subsets $\D \subseteq \N$.
\end{theorem}

As stated previously, a benefit of establishing a long variational inequality is that one can deduce a maximal inequality and norm convergence.

\begin{corollary} \label{cor:maximal}
Suppose $\{P_1,\dotsc,P_m\} \in \P$ and let $f \in L^p(X)$ with $p>1$. Then one has the maximal inequality \[ \left \| \sup_{N \in \N} |A_{N;X}^{P_1,\dotsc,P_m}f| \right\|_{L^p(X)} \leqsim_{P_1,\dotsc,P_m,m,p} \|f\|_{L^p(X)}, \] and the averages $A_{N;X}^{P_1,\dotsc,P_m}f$ converge both pointwise and in $L^p(X)$ norm as $N \to \infty$.
\end{corollary}

We remark that the norm convergence of $A_{N;X}^{P_1,\dotsc,P_m}$ is standard to prove and does not require the heavy machinery of Proposition \ref{prop:variationalineq}. Indeed, by the spectral theorem norm convergence follows from the pointwise convergence of the exponential sum \[ \frac{1}{N} \sum_{n=1}^N e(\xi_1 \lfloor P_1(n) \rfloor + \dotsb+\xi_m \lfloor P_m(n) \rfloor) \] for all $\xi_1,\dotsc,\xi_m$. But it is clear by (\ref{eqn:directform}) that we have convergence to $1_{\xi_1=\dotsb=\xi_m=0}$, so norm convergence follows.

We will in fact use Corollary \ref{cor:maximal} to improve on Theorem (\ref{thm:integervariation}) in some cases. For an appropriate subset $\P'$ of $\P$, which includes all families of the form $\{t^{c_1},\dotsc,t^{c_m} \}$ where $0<c_1<\dotsb<c_m$ are non-integer, we are able to prove full variational estimates (not just long variational ones).

\begin{theorem} \label{thm:integerfullvariation}
Suppose $\{P_1,\dotsc,P_m\} \in \P'$ and let $f \in \ell^p(\Z^m)$ with $p>1$. Then for any $r>2$ we have the variational inequality \begin{equation} \label{eqn:integerfullvariation}\|V^r(\tilde{A}_N^{P_1,\dotsc,P_m} f)_{N \in \N} \|_{\ell^p(\Z^m)} \leqsim_{P_1,\dotsc,P_m,m,p,r} \|f\|_{\ell^p(\Z^m)}. \end{equation}
\end{theorem}

By the same reductions as before and the Calder\'on transference principle again, Theorem \ref{thm:integervariation} implies full variational estimates on any measure preserving system.

\begin{corollary} \label{cor:fullvariation}
Suppose $\{P_1,\dotsc,P_m\} \in \P'$ and let $f \in L^p(X)$ with $p>1$. Then for any $r>2$ we have the variational inequality \[ \|V^r(A_N^{P_1,\dotsc,P_m}f)_{N \in \N}\|_{L^p(X)} \leqsim_{P_1,\dotsc,P_m,m,p,r} \|f\|_{L^p(X)}. \]
\end{corollary}

We will use Corollary \ref{cor:fullvariation} to obtain bounds on the rate of convergence of exponential sums of the form \[ \frac{1}{N} \sum_{n=1}^N e(\xi_1 \lfloor P_1(n) \rfloor + \dotsb + \lfloor P_m(n) \rfloor) \] whenever $P_1,\dotsc,P_m \in \P'$.

\subsection{Overview of Proof} The rest of the paper will be dedicated to proving Theorems \ref{thm:integervariation} and \ref{thm:integerfullvariation}. Working on the integer lattice will allow us to use harmonic analytic tools such as the Fourier transform, which will in turn allow us to use the circle method to analyse our averages. This application of the circle method is based on an analysis of the relevant exponential sums. One benefit of studying Hardy field functions is that we will only have one major arc---a box of suitable size centred at the origin.

Our operator is a convolution operator: $\tilde{A}_N^{P_1,\dotsc,P_m}f = K_N^{P_1,\dotsc,P_m}*f$ with \[ K_N^{P_1,\dotsc,P_m}
(x) = \int_{\T^m} m_{N;\Z}^{P_1,\dotsc,P_m}(\xi) e(x \cdot \xi) \,d\xi \] where \[ m_{N;\Z}^{P_1,\dotsc,P_m}(\xi_1,\dotsc,\xi_m) = \frac{1}{N} \sum_{n=1}^N 1_{n>\frac{N}{2}} e(\xi_1 \lfloor P_1(n) \rfloor + \dotsb + \xi_m \lfloor P_m(n) \rfloor). \] We will first prove that if $\delta \leq |m_{N;\Z}^{P_1,\dotsc,P_m}(\xi_1,\dotsc,\xi_m)|$ then either $N \leqsim_{P_1,\dotsc,P_m} \delta^{-O_{P_1,\dotsc,P_m}(1)}$ or $\|\xi_i\|_\T \leqsim_{P_1,\dotsc,P_m} \delta^{-O_{P_1,\dotsc,P_m}(1)} |P_i(N)|^{-1}$ for each $i \in \{1,\dotsc,m\}$, where $\|\xi\|_\T = \min_{n \in \Z} |\xi-n|$ is the distance from $\xi$ to the nearest integer. This result will define the relevant major arc where the symbol of our operator is large. Although we state and prove this exponential sum result in inverse form, one way to interpret it is to take $\delta = |m_{N;\Z}^{P_1, \dotsc, P_m}(\xi_1,\dotsc,\xi_m)|$, in which case one obtains \begin{equation} \label{eqn:directform} |m_{N;\Z}^{P_1,\dotsc,P_m}(\xi_1,\dotsc,\xi_m)| \leqsim_{P_1,\dotsc,P_m} N^{-O_{P_1,\dotsc,P_m}(1)} + \min_{1 \leq i \leq m} (\|\xi_i\|_\T |P_i(N)|)^{-O_{P_1,\dotsc,P_m}(1)}. \end{equation} In addition to providing the right definition for the major arc, we will also be able to use this inverse theorem to provide an $\ell^2(\Z^m)$ bound for functions supported on the minor arc which exhibits decay.

We will then turn to proving (\ref{eqn:integervariation}). The decay obtained by the single scale minor arc estimate coupled with the lacunary nature of $\D$ will allow us to prove that the $r$-variation of the minor arc piece of our operator is bounded on $\ell^p(\Z^m)$. At this point we can perform a Littlewood--Paley type decomposition and use tools such as vector valued square function bounds and L\'epingle's inequality to conclude the proof of Theorem \ref{thm:integervariation}.

Finally, we will turn to the proof of (\ref{eqn:integerfullvariation}). In proving full variational inequalities it is common to split the variation operator into a long variation piece and a short variation piece. The long variation of a sequence $(a_N)$ is simply defined as $V^r(a_N)_{N \in \D}$ where \[ \D = \{2^n : n \in \N\}, \] and the short variation is defined as \[ \left( \sum_{k=1}^\infty V^2(a_N)_{N \in \I_k}^2 \right)^\frac{1}{2}, \] where $\I_k = \{2^k,\dotsc,2^{k+1}\}$. Indeed, we have the following inequality \begin{equation} \label{eqn:longshort} V^r(a_N)_{N \in \N} \leqsim V^r(a_N)_{N \in \D} + \left( \sum_{k=1}^\infty V^2(a_N)_{N \in \I_k}^2 \right)^\frac{1}{2}. \end{equation}

Since Theorem \ref{thm:integervariation} will have been proved by this point and $\P' \subseteq \P$, it will suffice to only prove that the short variation operator is bounded on $\ell^p(\Z^m)$. This will be much easier to handle than the long variational inequality and will not require the circle method. We will instead use a Littlewood--Paley decomposition and split into the $p=2$ and $p \neq 2$ cases. Note that by this stage the proof of Corollary \ref{cor:maximal} will also be complete. This is relevant because we will need a maximal inequality in order to prove the short variational inequality.

\section{Applications}

In this section we outline some applications of our main results. To begin, we use Proposition \ref{prop:variationalineq} and Corollary \ref{cor:fullvariation} to study rates of convergence of exponential sums. Given $\delta>0$, an indexing set $\D$, and a sequence $(a_N)_{N \in \D}$, we define the jump counting function $\mathcal{N}_\delta(a_N)_{N \in \D}$ to be the supremum over all $J \in \N$ for which there exist times $N_0 \leq \dotsb \leq N_J$ with $N_j \in \D$ such that $|a_{N_{j+1}}-a_{N_j}| \geq \delta$ for all $j$. In other words, $\mathcal{N}_\delta(a_N)_{N \in \D}$ denotes the number of $\delta$-jumps of the sequence $(a_N)_{N \in \D}$. Jump counting functions and variation norms are connected by the inequality \begin{equation} \label{eqn:jumpcountingvariation} \delta \mathcal{N}_\delta(a_N)_{N \in \D}^\frac{1}{r} \leq V^r(a_N)_{N \in \D} \end{equation} which is easily seen.

\begin{corollary}
Let $\{P_1,\dotsc,P_m\} \in \P$. Given any $\lambda$-lacunary set $\D$ and any $\delta>0$, we have \[ \mathcal{N}_\delta \left( \frac{1}{N} \sum_{n=1}^N e(\xi_1 \lfloor P_1(n) \rfloor + \dotsb + \xi_m \lfloor P_m(n) \rfloor) \right)_{N \in \D} \leqsim_{\varepsilon,\lambda} \delta^{-(2+\varepsilon)} \] for any $\varepsilon>0$. If $\{P_1,\dotsc,P_m\} \in \P'$ then \[ \mathcal{N}_\delta \left( \frac{1}{N} \sum_{n=1}^N e(\xi_1 \lfloor P_1(n) \rfloor + \dotsb + \xi_m \lfloor P_m(n) \rfloor) \right)_{N \in \N} \leqsim_{\varepsilon} \delta^{-(2+\varepsilon)} \] for any $\varepsilon>0$. The implied constants are uniform in $\xi_1,\dotsc,\xi_m$.
\end{corollary}

\begin{proof}
Consider the torus $\T^m$ with the normalised Lebesgue measure, and define $T_i\:\T^m \to \T^m$ by $T_ix=x+\alpha_i e_i$ for some fixed $\alpha_i \in \T$, where $e_i$ is the $i$\textsuperscript{th} standard basis vector. Let $f(x)=e(\beta \cdot x)$ for fixed $\beta \in \T^m$. Then $\|f\|_{L^2(\T^m)}=1$ and \[ A_{N;\T}^{P_1,\dotsc,P_m}f(x) = e(\beta \cdot x) \frac{1}{N} \sum_{n=1}^N e(\beta \cdot (\alpha_1 \lfloor P_1(n) \rfloor+\dotsb+\lfloor P_m(n) \rfloor)). \] Given $\xi_1,\dotsc,\xi_m \in \T$, choose $\beta$ and $\alpha_i$ appropriately so that $\alpha_i\beta_i=\xi_i$. Then \[ \|V^r(A_{N;\T}^{P_1,\dotsc,P_m}f)_{N \in \D} \|_{L^2(\T^m)} = V^r \left( \frac{1}{N} \sum_{n=1}^N e(\xi_1 \lfloor P_1(n) \rfloor + \dotsb + \xi_m \lfloor P_m(n) \rfloor) \right)_{N \in \D}, \] so Proposition \ref{prop:variationalineq} implies \[ V^r \left(\frac{1}{N} \sum_{n=1}^N e(\xi_1 \lfloor P_1(n) \rfloor + \dotsb + \xi_m \lfloor P_m(n) \rfloor) \right)_{N \in \D} \leqsim_{r,\lambda} 1 \] for any $r>2$, uniformly in $\xi_1,\dotsc,\xi_m$. The first claim then follows by (\ref{eqn:jumpcountingvariation}). An analogous argument with $\D$ replaced by $\N$ and Corollary \ref{cor:fullvariation} instead of Proposition \ref{prop:variationalineq} yields the second claim.
\end{proof}

Next we give an application to equidistribution.

\begin{corollary}
Let $(X,\mu)$ be a probability space and let $T_1,\dotsc,T_m$ be pairwise commuting and ergodic transformations. Let $\{P_1,\dotsc,P_m\} \in \P$. Then for all $f \in L^p(X)$ with $p>1$ and $\mu$-almost every $x \in X$ we have \[ \lim_{N \to \infty} \frac{1}{N} \sum_{n=1}^N f(T_1^{\lfloor P_1(n) \rfloor} \dotsm T_m^{\lfloor P_m(n) \rfloor} x) = \int_X f \,d\mu. \] In particular, we have the equidistribution property \[ \lim_{N \to \infty} \frac{|\{1 \leq n \leq N : T_1^{\lfloor P_1(n) \rfloor} \dotsm T_m^{\lfloor P_m(n) \rfloor} x \in A \}|}{N} = \mu(A) \] for all measurable $A \subseteq X$ and $\mu$-almost every $x \in X$.
\end{corollary}

\begin{proof}
Note by (\ref{eqn:directform}) we have \[ \frac{1}{N} \sum_{n=1}^N e(\xi_1 \lfloor P_1(n) \rfloor + \dotsb + \lfloor P_m(n) \rfloor) \to 1_{\xi_1=\dotsb=\xi_m=0}. \] Thus by the spectral theorem and the assumption that $T_1,\dotsc,T_m$ are ergodic, it follows that \[ \lim_{N \to \infty} \frac{1}{N} \sum_{n=1}^N f(T_1^{\lfloor P_1(n) \rfloor} \dotsm T_m^{\lfloor P_m(n) \rfloor} x) = \int_X f \,d\mu \] for all $f \in L^\infty(X)$ with convergence taking place in $L^2(X)$ norm. Therefore, there is a subsequence $(N_j)$ along which one has \[ \lim_{j \to \infty} \frac{1}{N_j} \sum_{n=1}^{N_j} f(T_1^{\lfloor P_1(n) \rfloor} \dotsm T_m^{\lfloor P_m(n) \rfloor} x) = \int_X f \,d\mu \] pointwise almost everywhere. Combining this observation with Theorem \ref{thm:mainthm} proves that \[ \lim_{N \to \infty} \frac{1}{N} \sum_{n=1}^N f(T_1^{\lfloor P_1(n) \rfloor} \dotsm T_m^{\lfloor P_m(n) \rfloor} x) = \int_X f \,d\mu \] holds almost everywhere for all $f \in L^\infty(X)$. Since $L^\infty(X)$ is dense in $L^p(X)$, we obtain the same conclusion for $f \in L^p(X)$ with $p>1$ by the maximal inequality from Corollary (\ref{cor:maximal}). Specialising $f=1_A$ yields the equidistribution result.
\end{proof}

\section{Preliminaries}

\subsection{Hardy Field Functions} \label{subsec:hardyfield}

In this subsection we provide a definition of a Hardy field and discuss their properties. We will then formally describe the properties of $\P$ and $\P'$ that we will impose.

Consider the set of smooth functions $P\:(u,\infty) \to \R$, defined on some interval $(u,\infty)$, where $u \in \R \cup \{-\infty\}$. Call two such functions $P$ and $Q$ equivalent if there exists $T>0$ such that $P(t)=Q(t)$ for all $t \geq T$. This is an equivalence relation, and a \emph{germ at infinity} is simply an equivalence class under this notion of equivalence. The set of germs at infinity becomes a ring under the obvious operations $[P]+[Q]=[P+Q]$ and $[P][Q]=[PQ]$. By an abuse of notation we will denote $[P]$ by $P$. A subfield of the ring of germs at infinity which is closed under differentiation is called a \emph{Hardy field}. Common examples of Hardy fields include the field of constant functions on $\R$, the set $\R(x)$ of rational functions, and the so-called \emph{logarithmico-exponential} functions, which is the field generated (under the usual ring operations, differentiation, and function composition) by polynomials, the exponential function, and the logarithm.

Hardy field functions satisfy many nice properties. If $\H$ is a Hardy field and $P \in \H$ then $P$ is eventually positive, negative, or identically zero. This follows from the fact that $P$ must have a multiplicative inverse in $\H$ (as long as $P \neq 0$) and the Intermediate Value Theorem. As a consequence, $\H$ contains no periodic functions, and hence $\lim_{t \to \infty} P(t)$ always exists in $\R \cup \{\pm \infty\}$ for any $P \in \H$. One can therefore impose a total ordering $\prec$ on $\H$ by writing $P \prec Q$ if $Q-P$ is eventually positive. Moreover, note that for any distinct $P,Q \in \H$ either $\lim_{t \to \infty} \frac{P(t)}{Q(t)}$ or $\lim_{t \to \infty} \frac{Q(t)}{P(t)}$ is a finite constant. We therefore write $P(t) \sim_\H Q(t)$ if $\lim_{t \to \infty} \frac{P(t)}{Q(t)}=1$. Finally, one can define the \emph{type} of a function $P \in \H$ by \[ \tau(P) = \inf \{ \alpha \in \R \mid |P(t)| \leq t^\alpha \text{ eventually} \}. \] If $\tau(P)<\infty$ we say that $P$ is \emph{sub-polynomial}. Thus for every sub-polynomial $P \in \H$ there exists $\alpha \in \R$ such that for all $\varepsilon>0$ we have \begin{equation} \label{eqn:Hardybound} t^{\alpha-\varepsilon} \leq |P(t)| \leq t^{\alpha+\varepsilon} \end{equation} for all $t$ sufficiently large (depending on $\varepsilon$). We call $\tau(P)$ the \emph{growth rate} of $P$ in this case. We will call $P$ \emph{super-polynomial} if $\tau(P)>0$ and say that $P$ is of \emph{polynomial growth} if $0<\tau(P)<\infty$.

The type of a Hardy field function behaves well with respect to differentiation and inverses. Indeed, if $\tau(P)<\infty$ then $\tau(P') = \tau(P)-1$ (this can be seen by L'H\^opital's rule), and if additionally $\tau(P) \neq 0$ then $\tau(P^{-1}) = \tau(P)^{-1}$.

In this paper we will say a collection $\{P_1,\dotsc,P_m\}$ is an element of $\P$ if \begin{enumerate}[(i)]
    \item $P_1,\dotsc,P_m$ are of polynomial growth;
    \item $P_1,\dotsc,P_m$ are non-polynomial in the sense that for each $i \in \{1,\dotsc,m\}$ and $k \in \N$ we have $\lim_{t \to \infty} \frac{P_i(t)}{t^k} = 0$ or $\lim_{t \to \infty} \frac{t^k}{P_i(t)}=0$;
    \item $\tau(P_1),\dotsc,\tau(P_m)$ are distinct;
    \item $P_1,\dotsc,P_m$ satisfy the regularity conditions \begin{equation} \label{eqn:halforder} P_i\left(\frac{t}{2}\right) \sim P_i(t) \end{equation} and \begin{equation} \label{eqn:derivativeorder} tP'(t) \sim P(t) \end{equation} for all $t$ sufficiently large (depending on $P_1, \dotsc,P_m$).
\end{enumerate} We will say that $\{P_1,\dotsc,P_m\}$ is an element of $\P'$ if there exist non-integer $0<c_1<\dotsb<c_m$ such that $\lim_{t \to \infty} \frac{P_i(t)}{t^{c_i}}$ is a non-zero constant for each $i$ (thus $\tau(P_i)=c_i$). It is easily seen that $\P' \subseteq \P$. Note that if $\lim_{t \to \infty} \frac{P(t)}{t^c}$ is a non-zero constant then one can remove the dependence on $\varepsilon$ in (\ref{eqn:Hardybound}) in the sense that \begin{equation} \label{eqn:noepsHardybound} |P(t)| \sim t^c \end{equation} for all $t$ sufficiently large.

\subsection{Variation Norms}

We will use various standard facts about variation norms throughout the course of the paper. Let $\D$ be any countable ordered set and let $a,b\:\D \to \C$ be sequences of complex numbers. In this paper the ordered sets will be subsets of the natural numbers. As discussed previously, for $0<r <\infty$ the $V^r$ norm of the sequence $(a_N)$ is defined as \[ V^r(a_N)_{N \in \D} = \sup_{N \in \D} |a_N| + \sup_{J \in \N} \sup_{\substack{N_0<\dotsb<N_J \\ N_j \in \D}} \left( \sum_{j=0}^{J-1} |a_{N_{j+1}}-a_{N_j}|^r \right)^\frac{1}{r}. \] Then $V^r$ is a norm on the space of functions from $\D$ to $\C$ and is decreasing in $r$.

The first fact we will commonly use is that one can bound a $V^r$ norm by an $\ell^r$ norm. Since $\D$ is countable one can then bound an $\ell^r$ norm above by an $\ell^1$ norm. Thus we have
\begin{equation} V^r(a_N)_{N \in \D} \leqsim \left( \sum_{N \in \D} |a_N|^r \right)^\frac{1}{r} \leq \sum_{N \in \D} |a_N|. \label{eqn:variationvsellr} \end{equation}

In order to obtain lower bounds on the elements of the indexing set $\D$, the following fact will also be useful. One has \begin{equation} \label{eqn:orderedpartition} V^r(a_N)_{N \in \D} \leqsim V^r(a_N)_{N \in \D_1} + V^r(a_N)_{N \in \D_2} \end{equation} whenever $\D_1$ and $\D_2$ form a partition of $\D$ such that $N_1<N_2$ whenever $N_1 \in \D_1$ and $N_2 \in \D_2$.

\subsection{Notation}

We will use the convention that $\N = \{1,2,\dotsc\}$. For a real number $x$ we denote the greatest integer less than or equal to $x$ by $\lfloor x \rfloor$. For $x \in \R$ or $x \in \T$ we denote the complex exponential by $e(x) = e^{2\pi i x}$. All logarithms in this paper will be base 2. For a real number $\xi$ we define $\langle \xi \rangle = (1+|\xi|^2)^\frac{1}{2}$. Given $x,y \in \R^m$ we denote their inner product by $x \cdot y$.

Given a finite set $A$ and a function $f\:A \to \C$ we write \[ \E_{x \in A} f(x) = \frac{1}{|A|} \sum_{x \in A} f(x). \] Often, we will take $A = \{1, \dotsc, N\}$ or $A = \{-N, \dotsc, N\}$. We will denote these two sets by $[N]$ and $[[N]]$ respectively. If we are working in $\T$, then $[x]$ and $[[x]]$ will denote the intervals $[0,x]$ and $[-x,x]$ respectively (on identifying $\T$ with $[-\frac{1}{2},\frac{1}{2}])$. We denote the indicator function of a set $A$ by $1_A$. Note in particular that if $A$ is a subset of a finite set $B$ then \[ \E_{x \in B} 1_A(x) = \frac{|A|}{|B|} \] denotes the density of $A$ in $B$.

We will work with the Fourier transform on $\Z$ primarily, but also on $\R$. The Fourier transform and inverse Fourier transform of a function $f\:\Z \to \C$  will be denoted by $\hat{f}$ and $\check{f}$ respectively. For functions on $\R$ we will always denote the Fourier transform and its inverse by $\F_\R$ and $\F_\R^{-1}$ respectively.

Recall that the Haar measure on $\Z$ is the counting measure, and the pontyagrin dual $\Z$ is $\T$ with the normalised Lebesgue measure. The Haar measure on $\R$ is the Lebesgue measure and the pontyagrin dual of $\R$ is itself.

For quantities $X$ and $Y$ we write $X \leqsim Y$ or $X = O(Y)$ to mean that there is a constant $C>0$ such that $|X| \leq C |Y|$. If this constant $C$ depends on parameters $a_1,\dotsc,a_m$ we will denote this by $X \leqsim_{a_1,\dotsc,a_m} Y$ or $X = O_{a_1,\dotsc,a_m}(Y)$. We write $X \sim_{a_1,\dotsc,a_m} Y$ if both $X \leqsim_{a_1,\dotsc,a_m} Y$ and $Y \leqsim_{a_1,\dotsc,a_m} X$.

We now fix $m \in \N$ and Hardy field functions $P_1,\dotsc,P_m$ such that $\{P_1,\dotsc,P_m\} \in \P$. We will denote $\tau(P_i)$ by $c_i$ for each $i \in \{1,\dotsc,m\}$. By an abuse of notation we will denote $\Tilde{A}_N^{P_1,\dotsc,P_m}$ by $A_N$ and $m_{N;\Z}^{P_1,\dotsc,P_m}$ by $m_{N;\Z}$. We allow all implied constants to depend on the parameters $P_1,\dotsc,P_m,m,p,r,\lambda$ appearing in Theorem \ref{thm:integervariation} and will therefore suppress this dependence in implied constants from now on.

\section{Exponential Sums and a Single Scale Minor Arc Estimate}

The point of this section is to define the minor arcs for our operator $A_N$ and to show that if a function $f$ is Fourier supported on the minor arcs then the $\ell^2(\Z^m)$ norm of $A_Nf$ exhibits some decay. Both of these aims will be achieved by first proving an exponential sum bound.

\subsection{An Exponential Sum Estimate}

We begin by showing that if the exponential sum \[ m_{N;\Z} (\xi_1,\dotsc,\xi_m) = \E_{n \in [N]} e(\xi_1 \lfloor P_1(n) \rfloor + \dotsb + \xi_m \lfloor P_m(n) \rfloor) \] is larger in absolute value than some threshold $\delta \in (0,1]$, then either $N$ is small in the sense that $N \leqsim \delta^{-O(1)}$ or the coefficients $\xi_i$ lies in a small interval centred at the origin of width $O(\delta^{-O(1)} |P_i(N)|^{-1})$ for each $i$.

\begin{lemma} \label{lem:exponentialsum}
Let $N \geq 1$, $0<\delta \leq 1$, and $\xi_1, \dotsc, \xi_m \in \T$. Suppose that $\delta \leq |m_{N;\Z}(\xi_1, \dotsc,\xi_m)|$. Then either $N \leqsim \delta^{-O(1)}$ or we have $\|\xi_i\|_\T \leqsim \delta^{-O(1)} |P_i(N)|^{-1}$ for each $i \in \{1, \dotsc, m\}$.
\end{lemma}

One can also think of this as meaning that the tuple $(\xi_1,\dotsc,\xi_m)$ must lie in a small box of the form $\prod_{i=1}^m [[O(\delta^{-O(1)} |P_i(N)|^{-1})]]$. This view will determine the major arc (and thus the minor arc too) for our operator $A_N$, which will allow us to use the circle method in the next section.

To study the exponential sum $m_{N;\Z}$ we would like to exploit the smoothness of the Hardy field functions $P_1,\dotsc,P_m$. If the floor functions were not in place, we would be able to use the following lemma of Van der Corput that can be found in \cite{chan2010additive}.

\begin{lemma}[Van der Corput] \label{lem:vandercorputbound}
Let $j \geq 2$ be a natural number and $\frac{N}{2} \leq a<b \leq N$. Suppose $\eta \:[a,b] \to \R$ is $j$ times continuously differentiable and satisfies $0<\lambda \leq |\eta^{(j)}(t)| \leq h \lambda$ for all $t \in [a,b]$. Then \begin{equation} \left| \frac{1}{N} \sum_{n=a}^b e(\eta(n)) \right| \leqsim h \left( \lambda^{\frac{1}{J-2}}+N^{-\frac{2}{J}} + (\lambda N^j)^{-\frac{2}{J}} \right) \end{equation} where $J = 2^j$.
\end{lemma}

Lemma \ref{lem:vandercorputbound} will be key to obtaining principal major arc behaviour with the right bounds. Our first task is therefore to remove the floor function from $m_{N;\Z}(\xi_1,\dotsc,\xi_m)$. We will mimic the method utilised in \cite{chan2010additive} to do so.

\begin{theorem} \label{thm:removefloor}
Let $N \geq 1$, $0<\delta \leq 1$, and $\xi_1,\dotsc,\xi_m \in \T$. Suppose \[ \delta \leq | \E_{n \in [N]} 1_{n>\frac{N}{2}} e(\xi_1 \lfloor P_1(n) \rfloor+\dotsb+\xi_m \lfloor P_m(n) \rfloor) | \] for some $\xi_1, \dotsc, \xi_m$. Then either $N \leqsim \delta^{-O(1)}$ or there exist integers $k_1, \dotsc, k_m$ with $|k_i| \leqsim \delta^{-O(1)}$ such that \[ \delta^{O(1)} \leqsim | \E_{n \in [N]} 1_{n>\frac{N}{2}} e((\xi_1+k_1) P_1(n)+\dotsb+(\xi_m+k_m)P_m(n)) |. \]
\end{theorem}

\begin{proof}
Assume that $\xi_1, \dotsc,\xi_m \neq 0$. For notational ease we will also assume all sums in $n$ are supported on $n>\frac{N}{2}$. We begin by removing the floor function. Write \[ m_{N;\Z}(\xi_1, \dotsc,\xi_m) = \E_{n \in [N]} e(\xi_1 P_1(n)+\dotsb+\xi_m P_m(n)) \prod_{i=1}^m e(-\xi_i \{P_i(n)\} ). \] Expanding the function $x \mapsto e(-\xi \{x\})$ (for $\xi \neq 0$) into its Fourier series, for any $K \in \N$ one has \[ e(-\xi \{x\}) = a(\xi) \sum_{|k| \leq K} \frac{e(kx)}{\xi+k} + O\left(\log K \min \left(1,\frac{1}{K\|x\|_\T} \right) \right) \] where $|a(\xi)| \leq \|\xi\|_\T$.  Thus \begin{multline} \label{eqn:approximatebycontinuous} m_{N;\Z}(\xi_1,\dotsc,\xi_m) = \sum_{|k_1|,\dotsc,|k_m| \leq K} \E_{n \in [N]} e((\xi_1+k_1) P_1(n) + \dotsb + (\xi_m+k_m)P_m(n)) \prod_{i=1}^m \frac{a(\xi_i)}{(\xi_i+k_i)} \\ + O \left( K^{m-1} \log K \max_{1 \leq i \leq m} \E_{n \in [N]} \min \left(1,\frac{1}{K\|P_i(n) \|_\T} \right) \right). \end{multline} We first deal with the error term. We may further expand the function \[ \min \left(1,\frac{1}{K\|x\|_\T} \right) = \sum_{k \in \Z} b_k e(kx) \] into its Fourier series, where its coefficients satisfy \[ b_k \leqsim \min \left(\frac{\log K}{K}, \frac{K}{k^2} \right). \] Thus we may bound \[ \E_{n \in [N]} \min \left(1, \frac{1}{\|P_i(n)\|_\T} \right) \leqsim \frac{\log K}{K} + \sum_{0<|k| \leq K} \frac{\log K}{K} |\E_{n \in [N]} e(kP_i(n))| + \sum_{|k|>K} \frac{K}{k^2} |\E_{n \in [N]} e(kP_i(n))|. \] We take $K = A \delta^{-A}$ for a constant $A>0$ sufficiently large so that $\frac{\log K}{K}$ may be subtracted and absorbed into $\delta$. By the assumption (\ref{eqn:halforder}) made on $\P$ and Lemma \ref{lem:vandercorputbound}, taking $j=\lceil c_i \rceil+1$, we have \[ \sum_{0<|k| \leq K} \frac{\log K}{K} |\E_{n \in [N]} e(kP_i(n))| \leqsim  \log K( K^\frac{1}{2^j-2} |P_i^{(j)}(N)|^\frac{1}{2^j-2} + N^{-\frac{1}{2^{j-1}}} + |P_i^{(j)}(N)N^j|^{-\frac{1}{2^{j-1}}}) \] where we bounded $|k| \geq 1$ in the third term. Similarly, we may bound \[ \sum_{|k|>K} |\E_{n \in [N]} e(k P_i(n))| \leqsim K^\frac{1}{2^j-2} |P_i^{(j)}(N)|^\frac{1}{2^j-2} + N^{-\frac{1}{2^{j-1}}} + |P_i^{(j)}(N)N^j|^{-\frac{1}{2^{j-1}}}. \] With $c_i=t(P_i)$, we note that for any $\varepsilon_1,\varepsilon_2>0$ we have \[ K^\frac{1}{2^j-2} |P_i^{(j)}(N)|^\frac{1}{2^j-2} + N^{-\frac{1}{2^{j-1}}} + |P_i^{(j)}(N)N^j|^{-\frac{1}{2^{j-1}}} \leqsim_{\varepsilon_1, \varepsilon_2} K^\frac{1}{2^j-2} N^\frac{c_i-j+\varepsilon_1}{2^j-2} + N^{-\frac{1}{2^{j-1}}} + N^\frac{\varepsilon_2-c_i}{2^{j-1}}. \] Take $\varepsilon_1 = \frac{1}{2}$ and $\varepsilon_2=\frac{c_i}{2}$. Then taking $A>0$ is sufficiently large and ensuring that $N \geqsim \delta^{-O(1)}$, we see that we have a genuine error term in (\ref{eqn:approximatebycontinuous}), and thus \[ \delta \leqsim \sum_{|k_1|,\dotsc,|k_m| \leq K} |\E_{n \in [N]} e((\xi_1+k_1)P_1(n)+\dotsb+(\xi_m+k_m)P_m(n))| \prod_{i=1}^m \frac{|a(\xi_i)|}{|\xi_i+k_i|}. \] Note now that $|\xi_i+k_i| \geq \|\xi_i\|_\T$ since $k_i$ is an integer, and $|a(\xi_i)| \leq \|\xi_i\|_\T$. Thus we have \[ \delta \leqsim \sum_{|k_1|,\dotsc,|k_m| \leq K} |\E_{n \in [N]} e((\xi_1+k_1)P_1(n)+\dotsb+(\xi_m+k_m)P_m(n))|. \] We now apply the pigeonhole principle to deduce the existence of $k_1,\dotsc,k_m \in [[K]]$ such that \begin{equation} \label{eqn:continuousinequality} \delta^{O(1)} \leqsim |\E_{n \in [N]} e((\xi_1+k_1)P_1(n)+\dotsb+(\xi_m+k_m)P_m(n))|, \end{equation} completing the proof.
\end{proof}

If one tries to use the Van der Corput bound in Lemma \ref{lem:vandercorputbound}, it is easy to obtain the right major arc bounds if $m=1$ (see the proof of the next lemma to see why). The point of the next lemma is to show that if the $j$\textsuperscript{th} derivative of the function \[ t \mapsto (\xi_1+k_1) P_1(t) + \dotsb +(\xi_m+k_m) P_m(t) \] behaves like $(\xi_{i_0}+k_{i_0}) P_{i_0}^{(j)}(t)$ for a single $i_0$ on a sub-interval (in a sense reducing back to the $m=1$ case), then we obtain major arc behaviour not only for $\xi_{i_0}$, but also for \emph{every} $\xi_i$.

\begin{lemma} \label{lem:ontheorder}
Suppose \[ \delta \leq \left|\frac{1}{N} \sum_{n=a}^b 1_{n>\frac{N}{2}} e((\xi_1+k_1) P_1(n) +\dotsb+(\xi_m+k_m) P_m(n) ) \right| \] where $\frac{N}{2} \leq a<b \leq N$, $\xi_1, \dotsc, \xi_m \in \T$, and $k_1,\dotsc,k_m \in \Z$ satisfy $|k_1|, \dotsc, |k_m| \leqsim \delta^{-O(1)}$. Let $\eta(t) = (\xi_1+k_1) P_1(t)+\dotsb+(\xi_m+k_m)P_m(t)$. If there exists $c_{i_0}+1 \leq j \leqsim 1$ such that \begin{equation} \label{eqn:ontheorder} |\xi_{i_0}+k_{i_0}| \delta^{O(1)} |P_{i_0}^{(j)}(N)| \leqsim_{j} |\eta^{(j)}(t)| \leqsim_{j} |\xi_{i_0}+k_{i_0}|  |P_{i_0}^{(j)}(N)| \end{equation} for some $i_0 \in \{1, \dotsc, m\}$, then $\|\xi_i\|_\T \leqsim \delta^{-O(1)} |P_i(N)|^{-1}$ for all $i \in \{1, \dotsc, m\}$.
\end{lemma}

\begin{proof}
We proceed by induction. If $m=1$ then we have $|\eta^{(j)}(t)| \sim |\xi_1+k_1| |P_1^{(j)}(N)|$ by (\ref{eqn:halforder}) where $j = \lceil c_1 \rceil+1$. Then Lemma \ref{lem:vandercorputbound} implies \[ \delta \leqsim (|\xi_1+k_1||P_1^{(j)}(N)|)^\frac{1}{J-2} + N^{-\frac{2}{J}} + (|\xi_1+k_1||P_1^{(j)}(N)|N^j)^{-\frac{2}{J}} \] where $J = 2^j$. By bounding $|\xi_1+k_1| \leqsim \delta^{-O(1)}$ and ensuring $N \geqsim \delta^{-O(1)}$, we may absorb the first two terms on the right hand side so that \[ \delta \leqsim (|\xi_1+k_1||P_1^{(j)}(N)|N^j)^{-\frac{2}{J}}. \] Rearranging shows that $|\xi_1+k_1| \leqsim \delta^{-\frac{J}{2}} |P_1^{(j)}(N)|^{-1} N^{-j}$. Since $k_1$ is an integer we have $\|\xi_1\|_\T \leq |\xi_1+k_1|$, and thus by assumption (\ref{eqn:derivativeorder}) we have $\|\xi_1\|_\T \leqsim \delta^{-\frac{J}{2}} |P_1(N)|^{-1}$. This completes the proof if $m=1$, and if $m>1$ this provides the base case for the induction.

Now suppose $m>1$ and that the lemma holds for all elements of $\P$ with strictly less than $m$ elements. We apply Lemma \ref{lem:vandercorputbound} with the assumption (\ref{eqn:ontheorder}) to bound \begin{multline*} \left|\frac{1}{N} \sum_{n=a}^b 1_{n>\frac{N}{2}} e((\xi_1+k_1) P_1(n)+\dotsb+(\xi_m+k_m)P_m(n)) \right| \\ \leqsim \delta^{-O(1)} (|(\xi_{i_0}+k_{i_0})|\delta^{O(1)} |P_{i_0}^{(j)}(N)|)^\frac{1}{J-2} + N^{-\frac{2}{J}} + (|(\xi_{i_0}+k_{i_0}| \delta^{O(1)} |P_{i_0}^{(j)}(N)|N^j)^{-\frac{2}{J}}. \end{multline*} We may bound $|\xi_{i_0}+k_{i_0}| \leqsim \delta^{-O(1)}$. Then on ensuring that $N \geqsim \delta^{-O(1)}$ we may absorb the first two terms on the right hand side to bound \[ \delta^{O(1)} \leqsim (|(\xi_{i_0}+k_{i_0}| \delta^{O(1)} |P_{i_0}^{(j)}(N)| N^j)^{-\frac{2}{J}}. \] By rearranging and (\ref{eqn:derivativeorder}), we see that \[ |\xi_{i_0}+k_{i_0}| \leqsim \delta^{-O(1)} |P_{i_0}(N)|^{-1}. \] Since $k_{i_0}$ is an integer we have $\|\xi_{i_0}\|_\T \leq |\xi_{i_0}+k_{i_0}|$ and so $\|\xi_{i_0}\|_\T \leqsim \delta^{-O(1)} |P_{i_0}(N)|^{-1}$. Now let $\psi(n) = e((\xi_{i_0}+k_{i_0})P_{i_0}(n))$. For all $n$ we clearly have $|\psi(n)|=1$. Moreover, for all $n$ we have \[ \sum_{n=a}^b |\psi(n+1)-\psi(n)| \leqsim \sum_{n=a}^b |(\xi_{i_0}+k_{i_0})| |P_{i_0}(n+1)-P_{i_0}(n)| \leqsim \sum_{n=a}^b \delta^{-O(1)} |P_{i_0}(N)|^{-1} |P_{i_0}'(N)| \sim \delta^{-O(1)} \] by (\ref{eqn:derivativeorder}) again. Summation by parts therefore yields \[ \delta \leqsim \delta^{-O(1)} \sup_{a-1 \leq N' \leq b} \left| \frac{1}{N} \sum_{n=1}^{N'} 1_{n>\frac{N}{2}} e \left( \sum_{\substack{1 \leq i \leq m \\ i \neq i_0}} (\xi_i+k_i) P_i(n) \right) \right| \] and so there exists $N' \sim N$ such that \[ \delta^{O(1)} \leqsim \left| \frac{1}{N'} \sum_{n=1}^{N'} 1_{n>\frac{N}{2}} e \left( \sum_{\substack{1 \leq i \leq m \\ i \neq i_0}} (\xi_i+k_i) P_i(n) \right) \right|. \] Since $\{P_1,\dotsc,P_m\} \setminus\{P_{i_0}\} \in \P$, by induction we obtain $\|\xi_i\|_\T \leqsim \delta^{-O(1)} |P_{i}(N)|^{-1}$ for all $i \in \{1, \dotsc, m\}$.
\end{proof}

Recall that any Hardy field function $P$ eventually has a constant sign, so eventually $P = \pm P$. Thus any linear combination $\sum_{i=1}^m a_i P_i$ of Hardy field functions can be expressed as $\sum_{i=1}^m a_i' |P_i|$, a linear combination of their absolute values where $a_i' = \pm a_i$. Expressing it in this way means we only need to worry about the sign of $a_1',\dotsc,a_m'$ and not the sign of $P_1,\dotsc,P_m$ too. Despite the absolute values, because Hardy field functions have eventual constant sign we see that expressing linear combinations in this form does not violate any smoothness assumptions.

Let $\eta(t) = \sum_{i=1}^m (\xi_i+k_i)P_i(t)$. Then for all $t$ sufficiently large we can write \[ \eta^{(j)}(t) = \sum_{i=1}^m \alpha_{i,j}(\xi_i+k_i) |P_i^{(j)}(t)| \] for some coefficients $\alpha_{i,j} \in \{\pm 1\}$. If for some $j \geq \max_{1 \leq i \leq m} c_i+1$ the coefficients $\alpha_{i,j}(\xi_i+k_i)$ all have the same sign, then one can easily deduce \[ \max_{1 \leq i \leq m} |\xi_i+k_i||P_{i_0}^{(j)}(N)| \leqsim |\eta^{(j)}(t)| \leqsim \max_{1 \leq i \leq m} |\xi_i+k_i| |P_{i_0}^{(j)}(N)| \] for all $t \in [\frac{N}{2},N]$ hence (\ref{eqn:ontheorder}) holds and Lemma \ref{lem:ontheorder} can be applied. However, we cannot expect this to hold if at least one of the coefficients has an opposite sign to the rest. Thus our first task is to find a suitable interval on which this does hold. To facilitate this, we will find an exceptional set whose density in $[\frac{N}{2},N]$ can be made arbitrarily small. We will call any linear combination of Hardy field functions $Q_1,\dotsc,Q_k$ with $\{Q_1,\dotsc,Q_k\} \in \P$ a \emph{generalised polynomial}.

\begin{lemma} \label{lem:exceptionalset}
Let $N \geq 1$, $0<\delta<1$, and $A>0$. Let $\eta(t) = a_1 |P_1(t)|+\dotsb+a_m |P_m(t)|$ be a generalised polynomial defined on $[\frac{N}{2},N]$. If $E_{\eta,A}$ denotes the complement of the set of $t \in [\frac{N}{2},N]$ for which we have $\frac{\delta^A}{A} \max_{1 \leq i \leq m} |a_i|  |P_i(t)| \leq |\eta(t)| \leqsim \max_{1\leq i \leq m} |a_i| |P_i(t)|$ then $|E_{\eta,A}| \leqsim \frac{\delta^{O(A)}}{A} N$.
\end{lemma}

\begin{proof}
We proceed by induction. If $m=1$ then $E_{\eta,A}$ is empty and we are done. Now suppose $m>1$. Without loss of generality assume $a_m>0$. We define the correlation function \[ \CC_\eta(t) = \frac{\eta(t)}{\sum_{i=1}^m |a_i||P_i(t)|}. \] Note that if $|\CC_\eta(t)| \geq \frac{\delta^A}{A}$ for all $t \in [\frac{N}{2},N]$ then we have \[ \frac{\delta^A}{A} \max_{1 \leq i \leq m} |a_i| |P_i(t)| \leq |\eta(t)| \leqsim \max_{1 \leq i \leq m} |a_i||P_i(t)| \] for all $t \in [\frac{N}{2},N]$ so again $E_{\eta,A}$ is empty and we are done. Therefore, we can assume that there exists $t_0 \in [\frac{N}{2},N]$ such that $|\CC_\eta(t_0)| \leq \frac{\delta^A}{A}$. Since $\eta$ is a generalised polynomial, it is either monotonic on $[\frac{N}{2},N]$ or has $O(1)$ turning points. Thus the set of $t$ for which $|\CC_\eta(t)| \leq \frac{\delta^A}{A}$ is a union of $O(1)$ many intervals. Suppose $|\CC_\eta(t)|,|\CC_\eta(t_0)| \leq \frac{\delta^A}{A}$ for some $t$ lying in the same interval as $t_0$. Note that if $P_i$ is positive on $[\frac{N}{2},N]$ then since $c_i>0$ it follows that $P_i$ must be strictly increasing on $[\frac{N}{2},N]$, so $P_i'$ is also positive on $[\frac{N}{2},N]$. Similarly, if $P_i$ is negative then $P_i'$ is negative too. In other words, we have $\frac{d}{dt} |P_i(t)| = |P_i'(t)|$. Then by the Mean Value theorem we have \[ \frac{\delta^A}{A} \geq |\CC_\eta(t)-\CC_\eta(t_0)| = \left| \frac{\sum_{i=1}^m |a_i||P_i(t)| \sum_{j=1}^m a_j |P_j'(t)| - \sum_{i=1}^m a_i |P_i(t)| \sum_{j=1}^m |a_j| |P_j'(t)|}{(\sum_{i=1}^m |a_i| |P_i(\xi)|)^2} \right| |t-t_0| \] for some $\xi$ between $t$ and $t_0$, or \[ \frac{\delta^A}{A} \geq \left| \frac{\sum_{i \neq j} ( |a_i||P_i(t)|  a_j |P_j'(t)| -  a_i |P_i(t)|  |a_j| |P_j'(t)|) }{(\sum_{i=1}^m |a_i| |P_i(\xi)|)^2} \right| |t-t_0|. \] Since \[ |x|y-|y|x = \begin{cases} -2|x||y| & x>0,y<0 \\ 2|x||y| & x<0,y>0 \\ 0 & \text{otherwise} \end{cases} \] we see in fact this numerator is equal to \[ 2 \sum_{\substack{i \neq j \\ \omega_i \neq \omega_j}} \omega_j |a_i||P_i(t)||a_j||P_j'(t)|. \] where $\omega_j$ denotes the sign of $a_j$. Note the condition $|\CC(t_0)| \leq \frac{\delta^A}{A}$ implies that there exist $i,j \in \{1, \dotsc, m\}$ with $\omega_i \neq \omega_j$. In order to apply induction, we rewrite the numerator as \[ -\sum_{\omega_i=1} \sum_{\omega_j=-1} |a_i||a_j||P_i(t)||P_j'(t)| + \sum_{\omega_i=-1} \sum_{\omega_j=1} |a_i||a_j||P_i(t)||P_j'(t)| , \] which we can simplify to \begin{multline*} \sum_{\omega_i=1} \sum_{\omega_j=-1} |a_i||a_j| (|P_j(t)||P_i'(t)|-|P_i(t)||P_j'(t)|) \\ = \sum_{\omega_i=1} |a_i||P_i(t)| \sum_{\omega_j=-1} |a_j||P_j(t)| \left( \frac{|P_i'(t)|}{|P_i(t)|} - \frac{|P_j'(t)|}{|P_j(t)|} \right) . \end{multline*} Without loss of generality we will assume that $\max_{1 \leq k \leq m} |a_k| |P_k(\xi)| \neq |a_m||P_m(\xi)|$. The argument is analogous if this is not true with the role of $|a_1||P_1(\xi)|$, say, replacing that of $|a_m||P_m(\xi)|$. Then we decompose the numerator as \begin{multline*} \sum_{\substack{\omega_i=1 \\ i \neq m}} |a_i||P_i(\xi)| \sum_{\omega_j=-1} |a_j| |P_j(\xi)| \left( \frac{|P_i'(\xi)|}{|P_i(\xi)|} - \frac{|P_j'(\xi)|}{|P_j(\xi)|} \right) \\ + |a_m||P_m(\xi)| \sum_{\omega_j=-1} |a_j||P_j(\xi)| \left( \frac{|P_m'(\xi)|}{|P_m(\xi)|} - \frac{|P_j'(\xi)|}{|P_j(\xi)|} \right). \end{multline*} The condition $|\CC_\eta(\xi)| \leq \frac{\delta^A}{A}$ means that \[ |a_m||P_m(\xi)| =  \sum_{\omega_j=-1} |a_j||P_j(\xi)| - \sum_{\substack{\omega_i=1 \\ i \neq m}} |a_i||P_i(\xi)| + O \left( \frac{\delta^A}{A} \sum_{k=1}^m |a_k| |P_k(\xi)| \right). \] Thus we can rewrite our numerator as \begin{multline*} \sum_{\substack{\omega_i=1 \\ i \neq m}} |a_i||P_i(\xi)| \sum_{\omega_j=-1} |a_j||P_j(\xi)| \left( \frac{|P_i'(\xi)|}{|P_i(\xi)|} - \frac{|P_m'(\xi)|}{|P_m(\xi)|} \right) \\ + \sum_{\omega_j=-1} |a_j||P_j(\xi)| \sum_{\omega_{j'}=-1} |a_{j'}||P_{j'}(\xi)| \left( \frac{|P_m'(\xi)|}{|P_m(\xi)|} - \frac{|P_{j'}'(\xi)|}{|P_{j'}(\xi)|} \right)  \\ + O \left( \frac{\delta^A}{A} \sum_{k=1}^m |a_k| |P_k(\xi)| \sum_{\omega_j=-1} |a_j||P_j(\xi)| \left( \frac{|P_m'(\xi)|}{|P_m(\xi)|} - \frac{|P_j'(\xi)|}{|P_j(\xi)|} \right) \right). \end{multline*} Since $\frac{|P_i'(\xi)|}{|P_i(\xi)|} - \frac{|P_m'(\xi)|}{|P_m(\xi)|}$ is constant in $j$ we can change the order of summation in the first sum. Thus our numerator factorises as \begin{multline*} \left( \sum_{\omega_j=-1} |a_j||P_j(\xi)| \right) \left( \sum_{\substack{\omega_i=1 \\ i \neq m}} |a_i||P_i(\xi)| \left( \frac{|P_i'(\xi)|}{|P_i(\xi)|} - \frac{|P_m'(\xi)|}{|P_m(\xi)|} \right) \right. \\ + \left. \sum_{\omega_{j'}=-1} |a_{j'}||P_{j'}(\xi)| \left( \frac{|P_m'(\xi)|}{|P_m(\xi)|} - \frac{|P_{j'}'(\xi)|}{|P_{j'}(\xi)|} \right) \right) \\ + O \left(\frac{\delta^A}{A} \sum_{k=1}^m |a_k| |P_k(\xi)| \sum_{\omega_j=-1} |a_j||P_j(\xi)| \left( \frac{|P_m'(\xi)|}{|P_m(\xi)|} - \frac{|P_j'(\xi)|}{|P_j(\xi)|} \right) \right). \end{multline*} Let \[ \chi(\xi) = \xi \left( \sum_{\substack{\omega_i=1 \\ i \neq m}} |a_i||P_i(\xi)| \left( \frac{|P_i'(\xi)|}{|P_i(\xi)|} - \frac{|P_m'(\xi)|}{|P_m(\xi)|} \right) + \sum_{\omega_{j'}=-1} |a_{j'}||P_{j'}(\xi)| \left( \frac{|P_m'(\xi)|}{|P_m(\xi)|} - \frac{|P_{j'}'(\xi)|}{|P_{j'}(\xi)|} \right) \right). \] Then $\chi$ is a generalised polynomial of length strictly less than $m$. Indeed, we have \[ t\frac{P_i'(t)}{P_i(t)} \sim_\H \tau(P_i)>0 \] for all $i$, so for each $i \neq m$ we have \[ t|a_i| |P_i(t)| \left( \frac{|P_i'(t)|}{|P_i(t)|} - \frac{|P_m'(t)|}{|P_m(t)|} \right) \sim_\H (\tau(P_i)-\tau(P_m))|a_i||P_i(t)|. \]  With $B>0$ to be chosen later, by induction we have $|E_{\chi,B}| \leqsim \frac{\delta^{O(B)}}{B}N$. We decompose \[ E_{\eta,A} = (E_{\eta,A} \cap E_{\chi,B}) \cup (E_{\eta,A} \cap E_{\chi,B}^c). \] Since $|E_{\chi,B}| \leqsim \frac{\delta^{O(B)}}{B}N$ it suffices to estimate the measure of $E_{\eta,A} \cap E_{\chi,B}^c$. We can write this set as a union of $O(1)$ intervals. Suppose $t,t_0$ belong to the same interval. Then we have \[ \frac{\delta^A}{A} N \geqsim \left( \frac{|g(\xi)| |\sum_{ \omega_j=-1} |a_j| |P_j(\xi)|}{\xi (\max_{1 \leq k \leq m} |a_k| |P_k(\xi)|)^2} + O \left(\frac{\delta^A}{A} \frac{\xi \sum_{\omega_j=-1} |a_j||P_j(\xi)| \left( \frac{|P_m'(\xi)|}{|P_m(\xi)|} - \frac{|P_j'(\xi)|}{|P_j(\xi)|} \right)}{\xi \max_{1 \leq k \leq m} |a_k|\xi^{c_k}} \right) \right) |t-t_0|. \] Since $\xi$ lies in the same interval as $t$ and $t_0$, we have \[ \frac{\delta^B}{B} \max_{j \neq m} |a_j| |P_j(\xi)| \xi \left| \frac{|P_j'(\xi)|}{|P_j(\xi)|} - \frac{|P_m'(\xi)|}{|P_m(\xi)|} \right| \leq |g(\xi)|. \] In particular, we can bound \[ |g(\xi)| \geqsim \frac{\delta^B}{B} \xi \left| \frac{|P_k'(\xi)|}{|P_k(\xi)|} - \frac{|P_m'(\xi)|}{|P_m(\xi)|} \right| \max_{j \neq m} |a_j||P_j(\xi)| \] for some $k$. Since we allow constants to depend on $c_1, \dotsc, c_m$, we have \[ |g(\xi)| \geqsim \frac{\delta^B}{B} \max_{j \neq m} |a_j| |P_j(\xi)|. \] Recall $\max_{1 \leq k \leq m} |a_k||P_k(\xi)| \neq |a_m||P_m(\xi)|$, so $\max_{1 \leq k \leq m} |a_k||P_k(\xi)| = \max_{k \neq m} |a_k| |P_k(\xi)|$. We can bound \[ \xi \sum_{\omega_j=-1} |a_j||P_j(\xi)| \left( \frac{|P_m'(\xi)|}{|P_m(\xi)|} - \frac{|P_j'(\xi)|}{|P_j(\xi)|} \right) \leqsim \sum_{\omega_j=-1} |a_j| |P_j(\xi)|. \] Thus we obtain \[ \frac{\delta^A}{A} \geqsim \frac{\sum_{\omega_j=-1}|a_j| |P_j(\xi)|}{\xi \max_{1 \leq k \leq m} |a_k| |P_k(\xi)|} \left(\frac{\delta^B}{B}+O\left(\frac{\delta^A}{A} \right) \right)|t-t_0|. \] We have \[ \max_{1 \leq k \leq m} |a_k||P_k(\xi)| \in \left\{ \max_{\omega_i=-1} |a_i| |P_i(\xi)|, \max_{\omega_j=1} |a_j| |P_j(\xi)| \right\}. \] Moreover, the condition $|\CC_\eta(t)| \leq \frac{\delta^A}{A}$ implies \begin{multline*} \frac{\max_{\omega_j=-1} |a_j||P_j(\xi)|}{\max_{1 \leq k \leq m} |a_k||P_k(\xi)|} \geqsim \sum_{\omega_j=-1} \frac{|a_j||P_j(\xi)|}{\max_{1 \leq k \leq m} |a_k||P_k(\xi)|} \\ = \sum_{\omega_i=1} \frac{|a_i||P_i(\xi)|}{\max_{1 \leq k \leq m} |a_k||P_k(\xi)|} + O \left(\frac{\delta^A}{A} \right) \geq \frac{\max_{\omega_i=1} |a_i||P_i(\xi)|}{\max_{1 \leq k \leq m} |a_k||P_k(\xi)|} + O \left(\frac{\delta^A}{A} \right). \end{multline*} Since $\xi \in [\frac{N}{2},N]$, after bringing the $\xi$ in the denominator to the left hand side it follows that \[ \frac{\delta^A}{A} N \geqsim \left(1+O\left(\frac{\delta^A}{A} \right) \right) \left(\frac{\delta^B}{B}+O\left(\frac{\delta^A}{A} \right) \right)|t-t_0| \geqsim \left(\frac{\delta^B}{B}+O\left(\frac{\delta^A}{A} \right) \right)|t-t_0|. \] If we choose $B=cA$ with $c>0$ sufficiently small so that $\frac{\delta^B}{2B} \geqsim \frac{\delta^A}{A}$, we can conclude that \[ \frac{\delta^A}{A}N \geqsim \frac{\delta^B}{B} |t-t_0|, \] hence \[ |t-t_0| \leqsim \frac{\delta^{O(A)}}{A}N. \] It follows that $|E_{\eta,A}| \leqsim \frac{\delta^{O(A)}}{A}N$, as desired.
\end{proof}

With this lemma in hand, the major arc behaviour of our exponential sum follows readily.

\begin{proof}[Proof of Lemma \ref{lem:exponentialsum}]
By Theorem \ref{thm:removefloor}, as long as $N \geqsim \delta^{-O(1)}$, there exist $k_1, \dotsc, k_m \in \Z$ with $|k_i| \leqsim \delta^{-O(1)}$ such that \[ \delta^{O(1)} \leqsim \left| \frac{1}{N} \sum_{n=1}^N 1_{n>\frac{N}{2}} e((\xi_1+k_1) P_1(n)+\dotsb+(\xi_m+k_m) P_m(n)) \right|. \] Let $j=\max_{1\leq i \leq m} \lceil c_i \rceil+1$ and let $\eta$ denote the $j$\textsuperscript{th} derivative of $t \mapsto (\xi_1+k_1) P_1(t)+\dotsb+(\xi_m+k_m) P_m(t)$. Then note that \[ t^j P_i^{(j)}(t) \sim a_i P_i(t) \] by (\ref{eqn:derivativeorder}), where $a_i$ is a constant depending only on $c_1,\dotsc,c_m$. Thus $t^j\eta(t)$ is a generalised polynomial, so by Lemma \ref{lem:exceptionalset} for any $A>0$ there is a set $E_A$ with $|E_A| \leqsim \frac{\delta^{O(A)}}{A}$ such that \[ \frac{\delta^A}{A} \max_{1 \leq i \leq m} |\xi_i+k_i||t^j P_i^{(j)}(t)| \leq |t^j\eta(t)| \leqsim \max_{1 \leq i \leq m} |\xi_i+k_i|| t^j P_i^{(j)}(t)| \] holds for all $t \not \in E_A$. Partition $[\frac{N}{2},N]$ into intervals of length $\frac{\delta^{O(B)}}{B}N$, with $A$ sufficiently large compared to $B$. Then by the pigeonhole principle there exists a single interval $I$ with $\frac{\delta^B}{B}N \leqsim |I| \leq N$ an $i$ for which we have \[ \frac{\delta^A}{A}  |\xi_i+k_i|| P_i^{(j)}(t)| \leq |\eta(t)| \leqsim  |\xi_i+k_i|| P_i^{(j)}(t)| \] for all $t \in I$, and such that \[ \delta^{O(1)} \leqsim \E_{n \in I} e((\xi_1+k_1) P_1(n)+\dotsb+(\xi_m+k_m) P_m(n)). \] We therefore have \[ \frac{\delta^A}{A}|\xi_i+k_i||P_{i}^{(j)}(N)| \leqsim |\eta(t)| \leqsim |\xi_i+k_i||P_{i}^{(j)}(N)| \] for all $t \in I$, so the result follows by Lemma \ref{lem:ontheorder}.
\end{proof}

\subsection{The Single Scale Minor Arc Estimate}

Note that by (\ref{eqn:variationvsellr}), if for some $p$ and $f$ we can prove that $\|A_Nf\|_{\ell^p(\Z^m)} \leqsim a_N \|f\|_{\ell^p(\Z^m)}$ with $\sum_{N \in \D} a_N \leqsim 1$ then by the triangle inequality in $\ell^p(\Z^m)$ it will follow that $\|V^r(A_Nf)_{N \in \D}\|_{\ell^p(\Z^m)} \leqsim \|f\|_{\ell^p(\Z^m)}$. In general we have the following trivial estimate \begin{equation} \label{eqn:trivialestimate} \|A_Nf\|_{\ell^p(\Z^m)} \leq \|f\|_{\ell^p(\Z)}, \end{equation} which is easily seen by the triangle inequality in $\ell^p(\Z^m)$ and translation invariance, so we will require more information about $p$ or $f$.

Recall that $A_N$ is a convolution operator with kernel \[ K_N(x) = \int_{\T^m} m_{N;\Z}(\xi) e(x \cdot \xi) \,d\xi. \] By Lemma \ref{lem:exponentialsum} we know that $m_{N;\Z}$ is only large on a small box centred at the origin. This heuristic suggests that for $p=2$ and $f$ with a Fourier transform vanishing on a suitable box our operator $A_Nf$ should be small. Indeed, in this subsection we can easily prove the following theorem using Lemma \ref{lem:exponentialsum}.

\begin{theorem} \label{thm:minorarcestimate}
Let $N \geq 1$ and $l \in \N$. Suppose $f \in \ell^2(\Z^m)$ has a Fourier transform which vanishes on $\prod_{i=1}^m [[ 2^l |P_i(N)|^{-1} ]]$. Then \[ \|A_Nf\|_{\ell^2(\Z^m)} \leqsim (N^{-O(1)}+2^{-O(l)}) \|f\|_{\ell^2(\Z^m)}. \]
\end{theorem}

\begin{proof}
First suppose that $N \geqsim 2^{O(l)}$. Let $\M_N = \prod_{i=1}^m [[2^l|P_i(N)|^{-1}]]$. Then by Plancherel's theorem, the support assumption on $\hat{f}$, H\"older's inequality, and Plancherel's theorem again we have \[ \|A_Nf\|_{\ell^2(\Z^m)} = \|m_{N;\Z} 1_{\M_N^c} \hat{f}\|_{L^2(\T^m)} \leq \|m_{N;\Z} 1_{\M_N^c}\|_{L^\infty(\T^m)} \|f\|_{\ell^2(\Z^m)}, \] so it now suffices to prove that \[ \|m_{N;\Z} 1_{\M_N^c}\|_{L^\infty(\T^m)} \leqsim 2^{-O(l)}. \] We now use Lemma \ref{lem:exponentialsum} in the contrapositive: if there exists $i$ such that $\|\xi_i\|_\T \geq \delta^{-1} |P_i(N)|^{-1}$ then $|m_{N;\Z}(\xi_1,\dotsc,\xi_m)| \leqsim \delta^{O(1)}$. Taking $\delta = 2^{-l}$ and noting that if $(\xi_1,\dotsc,\xi_m) \in \M_N^c$ then there exists $i$ such that $\|\xi_i\|_\T \geq \delta^{-1} |P_i(N)|^{-1}$ we obtain the desired bound. Let $l_0$ be the maximal $l$ such that $N \geqsim 2^{O(l)}$, and note that $2^{l_0} \sim N^{O(1)}$. For any $l$ such that $N \leqsim 2^{O(l)}$, if $\hat{f}$ vanishes on $\prod_{i=1}^m [[2^l|P_i(N)|^{-1}]]$ then $\hat{f}$ vanishes on $\prod_{i=1}^m [[2^{l_0} |P_i(N)|^{-1}]]$ so we obtain \[ \|A_Nf\|_{\ell^2(\Z^m)} \leqsim 2^{-O(l_0)} \|f\|_{\ell^2(\Z^m)} \sim N^{-O(1)} \|f\|_{\ell^2(\Z^m)} \] thus completing the proof.
\end{proof}

\section{Long Variation}

In this section we turn to the proof of Theorem \ref{thm:integervariation}. As discussed previously, we will use the circle method to decompose $f$ into a major arc piece and a minor arc piece.

We will need three constants $C_0,C_1,C_2 \in \N$ which will be determined by the end of the section. We allow them to depend on $P_1,\dotsc,P_m,m,p,r,\lambda$, and allow $C_2$ to depend on $C_0,C_1$.

\subsection{Reducing to Major Arcs}

Note that by (\ref{eqn:orderedpartition}), (\ref{eqn:variationvsellr}), and (\ref{eqn:trivialestimate}) we may assume that $N \geq C_2$ for all $N \in \D$. Thus it now suffices to prove that \begin{equation} \|V^r(A_Nf)_{N \in \D} \|_{\ell^p(\Z^m)} \leqsim_{C_2} \|f\|_{\ell^p(\Z^m)} \end{equation} for all $f \in \ell^p(\Z^m)$.

Now take $l_N = \lfloor C_0 \log \log N \rfloor$ and define \[ \widehat{f_{\M_N}}(\xi_1,\dotsc,\xi_m) = \hat{f}(\xi_1, \dotsc, \xi_m) \prod_{i=1}^m \Psi_{\leq 2^{l_N} |P_i(N)|^{-1}} (\xi_i) \] and $f_{\M_N^c} = f - f_{\M_N}$. Thus $\widehat{f_{\M_N}}$ is supported on $\prod_{i=1}^m [[ 2^{l_N+1} |P_i(N)|^{-1} ]]$ and $\widehat{f_{\M_N^c}}$ vanishes on $\prod_{i=1}^m [[ 2^{l_N} |P_i(N)|^{-1} ]]$. Thus by Theorem \ref{thm:minorarcestimate} we have \begin{equation} \|A_N f_{\M_N^c}\|_{\ell^2(\Z^m)} \leqsim_{C_2} (\log N)^{-O(C_0)} \|f\|_{\ell^2(\Z^m)}. \end{equation} Interpolation with (\ref{eqn:trivialestimate}) yields \begin{equation} \|A_Nf\|_{\ell^p(\Z^m)} \leqsim_{C_2} (\log N)^{-O(C_0)} \|f\|_{\ell^p(\Z^m)} \end{equation} for all $p>1$. Taking $C_0$ sufficiently large, but $O(1)$, we obtain \begin{equation} \|A_Nf\|_{\ell^p(\Z^m)} \leqsim_{C_2} (\log N)^{-10} \|f\|_{\ell^p(\Z^m)}. \end{equation} It follows that \[ \|V^r(A_N f_{\M_N^c})_{N \in \D}\|_{\ell^p(\Z^m)} \leqsim \sum_{N \in \D} \|A_N f_{\M_N^c}\|_{\ell^p(\Z^m)} \leqsim_{C_2} \|f\|_{\ell^p(\Z^m)} \sum_{N \in \D} (\log N)^{-10}, \] and since $\D$ is $\lambda$-lacunary we have $\sum_{N \in \D} (\log N)^{-10}\leqsim 1$. By the linearity of $A_N$ and the triangle inequality in $V^r$ and $\ell^p(\Z^m)$ it now suffices to prove that \begin{equation} \label{eqn:majorarcbound} \|V^r(A_N f_{\M_N})_{N \in \D} \|_{\ell^p(\Z^m)} \leqsim_{C_2} \|f\|_{\ell^p(\Z^m)}. \end{equation}

\subsection{The Major Arc Bound}

Note that \[ A_N f_{\M_N} = f*K_{\M_N} \] where \[ K_{\M_N}(x) = \int_{\T^m} e(x \cdot \xi) \E_{n \in [N]} 1_{n>\frac{N}{2}} e(\xi_1 \lfloor P_1(n) \rfloor + \dotsb + \xi_m \lfloor P_m(n) \rfloor) \prod_{i=1}^m \Psi_{\leq 2^{l_N} |P_i(N)|^{-1} } (\xi_i) \,d\xi. \]

Fix a cutoff function $\Psi \in C_c^\infty(\R,[0,1])$ which is non-negative, even, supported on $[-1,1]$, and equal to 1 on $[-\frac{1}{2},\frac{1}{2}]$. We allow all implied constants to depend on $\Psi$. For a real number $x$, we define \[ \Psi_{\leq x} = \Psi\left(\frac{\cdot}{2^{\lceil \log x \rceil}}\right). \] Then $\Psi_{\leq x}$ is supported on $\{ |\xi| \leq 2^{\lceil \log x \rceil} \}$ with $2^{\lceil \log x \rceil} \sim x$. Define \[ \Phi_{N,l}(\xi_1,\dotsc,\xi_m) = \begin{cases} \prod_{i=1}^m \Psi_{\leq 2^l|P_i(N)|^{-1}}(\xi_i) - \prod_{i=1}^m \Psi_{\leq 2^{l-1}|P_i(N)|^{-1}}(\xi_i), & l>-C_1 \\ \prod_{i=1}^m \Psi_{\leq 2^{-C_1} |P_i(N)|^{-1}}(\xi_i), & l=-C_1. \end{cases} \] Then we can decompose \[ \prod_{i=1}^m \Psi_{\leq N^{-c_i} 2^{l_N}} (\xi_i) = \sum_{-C_1 \leq l \leq l_N} \Phi_{N,l}(\xi_1,\dotsc,\xi_m). \] We let \[ f_{N,l}(x) = \int_{\T^m} \hat{f}(\xi) e(x \cdot \xi) \Phi_{N,l}(\xi) \,d\xi \] and define $\D_l = \{N \in \D \mid l \leq l_N \}$. By the triangle inequality in $V^r$ and $\ell^p(\Z)$ we can bound \[ \| V^r(\Tilde{A}_Nf_{\M_N})_{N \in \D} \|_{\ell^p(\Z^m)} \leqsim_{C_2} \sum_{l \geq -C_1} \|V^r(\Tilde{A}_N f_{N,l})_{N \in \D_l} \|_{\ell^p(\Z^m)}. \] Thus it suffices to prove the following.

\begin{theorem}
Let $l \geq -C_1$. Then \[ \|V^r(\Tilde{A}_N f_{N,l})_{N \in \D_l} \|_{\ell^p(\Z^m)} \leqsim_{C_2} \langle l \rangle^{O(1)} 2^{-O(l) 1_{p=2}} \|f\|_{\ell^p(\Z^m)}. \]
\end{theorem}

There are some cases we need to distinguish between here. Firstly we have cases depending on whether or not $p=2$ and secondly depending on whether or not $l=-C_1$. We can easily handle the case $p=2$ and $l>-C_1$ using Theorem \ref{thm:minorarcestimate}.

\begin{proposition}
Suppose $l>-C_1$. Then \[ \|V^r(A_Nf_{N,l})_{N \in \D} \|_{\ell^2(\Z^m)} \leqsim_{C_2} 2^{-O(l)} \|f\|_{\ell^2(\Z^m)}. \]
\end{proposition}

\begin{proof}
By (\ref{eqn:variationvsellr}) and Fubini's theorem we may bound \[ \|V^r(A_N f_{N,l})_{N \in \D} \|_{\ell^2(\Z^m)} \leqsim \left( \sum_{N \in \D} \|A_N f_{N,l}\|_{\ell^2(\Z^m)}^2 \right)^\frac{1}{2}. \] Since $l>-C_1$, observe that $\hat{f}_{N,l}$ vanishes on $\prod_{i=1}^m [[ O(2^l |P_i(N)|^{-1}) ]]$. Thus by Theorem \ref{thm:minorarcestimate} we have \[ \|A_N f_{N,l} \|_{\ell^2(\Z^m)} \leqsim 2^{-O(l)} \|f_{N,l}\|_{\ell^2(\Z^m)}. \] Therefore, \[ \|V^r(A_N f_{N,l})_{N \in \D} \|_{\ell^2(\Z^m)} \leqsim 2^{-O(l)} \left( \sum_{N \in \D} \|f_{N,l}\|_{\ell^2(\Z^m)}^2 \right)^\frac{1}{2}. \] By Plancherel's theorem, Fubini's theorem, and Plancherel's theorem again, it suffices to show that \[ \left \| \sum_{N \in \D} \Phi_{N,l}(\xi) \right\|_{L^\infty(\T^m)} \leqsim 1. \] But this follows from the support of $\Phi_{N,l}$ and the lacunary nature of $\D$.
\end{proof}

To handle the rest of the cases we would like to replace the multiplier $m_{N;\Z}$ with its continuous counterpart \[ m_{N;\R}(\xi_1, \dotsc, \xi_m) = \frac{1}{N} \int_\frac{N}{2}^N e(\xi_1 P_1(t) + \dotsb + \xi_m P_m(t)) \,dt. \] We define \[ K_{m_{N;\R}}^l(y) = \int_{\T^m} m_{N;\R}(\xi) \Phi_{N,l}(\xi) e(y \cdot \xi) \,d\xi \] for all $y \in \Z^m$. Replacing these multipliers will be facilitated by the following proposition.

\begin{proposition} \label{prop:replacemultiplier}
For any $l \geq -C_1$ we have \[ \| A_Nf_{N,l} - f*K_{m_{N;\R}}^l\|_{\ell^p(\Z^m)} \leqsim_{C_2} 2^{O(l)} N^{-O(1)} \|f\|_{\ell^p(\Z^m)}. \]
\end{proposition}

\begin{proof}
Note that

\begin{multline*} f * K^l_{m_{N;\R}}(x) = \frac{1}{N} \int_{\frac{N}{2}}^N \sum_{y \in \Z^m} f(x-y) 
 \\ \left( \prod_{i=1}^m  \frac{1}{2^{-l} |P_i(N)|} \F_\R^{-1} \Psi \left(\frac{y_i-P_i(t)}{2^{-l} |P_i(N)|} \right) - 1_{l>-C_1} \prod_{i=1}^m  \frac{1}{2^{-(l-1)} |P_i(N)|} \F_\R^{-1} \Psi \left(\frac{y_i-P_i(t)}{2^{-(l-1)} |P_i(N)|} \right) \right) \,dt. \end{multline*} By Minkowski's inequality and Young's inequality we therefore have \begin{multline*} \|f * K_{m_{N;\R}} \|_{\ell^p(\Z^m)} \leq \|f\|_{\ell^p(\Z^m)} \frac{1}{N} \int_\frac{N}{2}^N \\  \sum_{y\in \Z^m} \left| \prod_{i=1}^m  \frac{1}{2^{-l} |P_i(N)|} \F_\R^{-1} \Psi \left(\frac{y_i-P_i(t)}{2^{-l} |P_i(N)|} \right) - 1_{l>-C_1} \prod_{i=1}^m  \frac{1}{2^{-(l-1)} |P_i(N)|} \F_\R^{-1} \Psi \left(\frac{y_i-P_i(t)}{2^{-(l-1)} |P_i(N)|} \right) \right|. \end{multline*} Since $\Psi$ is Schwartz we have \begin{multline*} \sum_{y\in \Z^m} \left| \prod_{i=1}^m  \frac{1}{2^{-l} |P_i(N)|} \F_\R^{-1} \Psi \left(\frac{y_i-P_i(t)}{2^{-l} |P_i(N)|} \right) - 1_{l>-C_1} \prod_{i=1}^m  \frac{1}{2^{-(l-1)} |P_i(N)|} \F_\R^{-1} \Psi \left(\frac{y_i-P_i(t)}{2^{-(l-1)} |P_i(N)|} \right) \right| \\ \leqsim 2^{O(l)}. \end{multline*} Therefore, by the triangle inequality in $\ell^p(\Z^m)$ and (\ref{eqn:trivialestimate}) it follows that \[ \|A_N f_{N,l} - f*K_{m_{N;\R}}^l \|_{\ell^p(\Z^m)} \leqsim 2^{O(l)} \|f\|_{\ell^p(\Z^m)}. \] By interpolation it now suffices to establish the $p=2$ case of the proposition. By Plancherel's theorem and H\"older's inequality it suffices to estimate \[ \|\Phi_{N,l}(m_{N;\Z}-m_{N;\R}) \|_{L^\infty(\T^m)} \leqsim 2^{O(l)} N^{-O(1)}. \] If $\xi$ is in the support of $\Phi_{N,l}$ then $|\xi_i| \leqsim 2^l |P_i(N)|^{-1}$ for each $i$. We can then express \begin{multline*} m_{N;\Z}(\xi) - m_{N;\R}(\xi) \\ = \frac{1}{N} \sum_{n=1}^N 1_{n>\frac{N}{2}} \int_{n-1}^n e(-\xi_1 \lfloor P_1(n) \rfloor - \dotsb-\xi_m \lfloor P_m(n) \rfloor) - e(-\xi_1  P_1(t)  - \dotsb-\xi_m  P_m(t)) \,dt. \end{multline*} Thus by the triangle inequality it suffices to show that \[ |e(-\xi_1 \lfloor P_1(n) \rfloor - \dotsb-\xi_m \lfloor P_m(n) \rfloor) - e(-\xi_1  P_1(t)  - \dotsb-\xi_m  P_m(t))| \leqsim 2^{O(l)} N^{-O(1)}. \] By Taylor expansion and the triangle inequality, we have \[ |e(-\xi_1 \lfloor P_1(n) \rfloor - \dotsb-\xi_m \lfloor P_m(n) \rfloor) - e(-\xi_1  P_1(t)  - \dotsb-\xi_m  P_m(t))| \leq \sum_{i=1}^m 2^l |P_i(N)|^{-1} |\lfloor P_i(n) \rfloor - P_i(t)|. \] Since $P_i$ is a Hardy field function, it follows that \[ |P_i(N)|^{-1} |\lfloor P_i(n) \rfloor - P_i(t)| \leqsim_\varepsilon N^{-1+\varepsilon} \] for all $\varepsilon>0$. Taking $\varepsilon=\frac{1}{2}$, say, the result follows.
\end{proof}

Since $l \leq l_N$ for all $N \in \D_l$ we see that $N \geq \max( 2^{2^{l/C_0}}, C_2)$. It follows that \[ 2^{-O(l)} \sum_{N \in \D_l} N^{-O(1)} \leqsim_{C_2} \langle l \rangle^{O(1)}  2^{-O(l) 1_{p=2}}. \] Therefore, by (\ref{eqn:variationvsellr}) and Proposition \ref{prop:replacemultiplier} it follows that \[ \| V^r(A_N f_{N,l} - f * K_{m_{N;\R}}^l)_{N \in \D_l} \|_{\ell^p(\Z^m)} \leqsim_{C_2} \langle l \rangle^{O(1)} 2^{-O(l) 1_{p=2}} \|f\|_{\ell^p(\Z^m)}. \] By the triangle inequality in $V^r$ and $\ell^p(\Z^m)$, it therefore suffices to prove that \[ \|V^r( f*K_{m_{N;\R}}^l)_{N \in \D_l} \|_{\ell^p(\Z^m)} \leqsim_{C_2} \langle l \rangle^{O(1)} 2^{-O(l)1_{p=2}} \|f\|_{\ell^p(\Z^m)}. \]

The benefit of replacing $m_{N;\Z}$ with $m_{N;\R}$ is we can change variables to write \[ m_{N;\R}(\xi_1,\dotsc, \xi_m) = \int_\frac{1}{2}^1 e(\xi_1 P_1(Nt) + \dotsb + \xi_m P_m(Nt)) \,dt. \] Therefore, we have \[ f*K_{m_{N;\R}}^l = \int_{\frac{1}{2}}^1 f * \F_{\R^m}^{-1} \Phi_{N,l}(\cdot-(P_1(Nt),\dotsc,P_m(Nt)) \,dt. \] We can apply Minkowski's inequality to bound \[ \|V^r(f*K_{m_{N;\R}}^l)_{N \in \D_l} \|_{\ell^p(\Z^m)} \leq \int_\frac{1}{2}^1 \|V^r(f*\F_{\R^m}^{-1} \Phi_{N,l}(\cdot-(P_1(Nt),\dotsc,P_m(Nt)))_{N \in \D_l} \|_{\ell^p(\Z^m)} \,dt. \] Thus it now suffices to prove the following proposition.

\begin{proposition}
Suppose that either $p \neq 2$ and $l>-C_1$ or that $l=-C_1$. Then for each $t \in [\frac{1}{2},1]$ we have \[ \|V^r(f*\F_{\R^m}^{-1} \Phi_{N,l}(\cdot-(P_1(Nt),\dotsc,P_m(Nt)))_{N \in \D_l} \|_{\ell^p(\Z^m)} \leqsim_{C_2} \langle l \rangle^{O(1)} \|f\|_{\ell^p(\Z^m)} . \]
\end{proposition}

\begin{proof}
First suppose $p \neq 2$ and $l>-C_1$. Then using (\ref{eqn:variationvsellr}) we bound \begin{multline*} \| V^r(f*\F_{\R^m}^{-1} \Phi_{N,l}(\cdot-(P_1(Nt),\dotsc,P_m(Nt)))_{N \in \D_l} \|_{\ell^p(\Z^m)} \\ \leqsim \left\| \left( \sum_{N \in \D_l} |f*\F_{\R^m}^{-1} \Phi_{N,l}(\cdot-(P_1(Nt),\dotsc,P_m(Nt))|^2 \right)^\frac{1}{2} \right\|_{\ell^p(\Z^m)}. \end{multline*} By Khintchine's inequality it suffices to prove that \[ \left\| \sum_{N \in \D_l} \varepsilon_N f * \F_{\R^m}^{-1} \Phi_{N,l}(\cdot-(P_1(Nt),\dotsc,P_m(Nt)) \right\|_{\ell^p(\Z^m)} \leqsim \langle l \rangle^{O(1)} \|f\|_{\ell^p(\Z^m)} \] for all 1-bounded $\varepsilon_N \in \C$. By Plancherel's theorem we have \begin{multline*} \left\| \sum_{N \in \D_l} \varepsilon_N f * \F_{\R^m}^{-1} \Phi_{N,l}(\cdot-(P_1(Nt),\dotsc,P_m(Nt)) \right\|_{\ell^2(\Z^m)}^2 \\ = \int_{\T^m} |\hat{f}|^2 \left| \sum_{N \in \D} \varepsilon_N  \Phi_{N,l}(\xi_i-P_i(Nt)) \right|^2 \,d\xi. \end{multline*} Since $l>-C_1$ and $\D_l$ is lacunary, it follows that \[ \left\| \sum_{N \in \D_l} \varepsilon_N f * \F_{\R^m}^{-1} \Phi_{N,l}(\cdot-(P_1(Nt),\dotsc,P_m(Nt)) \right\|_{\ell^2(\Z^m)} \leqsim \|f\|_{\ell^2(\Z^m)}. \] By duality and the Marcinkiewicz interpolation theorem it suffices to establish the weak type bound \[ \left| \left\{ x \in \Z^m : \left| \sum_{N \in \D_l} \varepsilon_N (f * \F_{\R^m}^{-1} \Phi_{N,l}(\cdot-(P_1(Nt),\dotsc,P_m(Nt)))(x) \right| \geq \alpha \right\} \right| \leqsim_{C_2} \langle l \rangle^{O(1)} \frac{\|f\|_{\ell^1(\Z^m)}}{\alpha}. \] By performing a Calder\'on--Zygmund decomposition on $\Z^m$, the standard proof that Calder\'on--Zygmund operators are weak-(1,1) reduces to showing that \[ \sum_{x \in (100 Q_j)^c} \left| \sum_{N \in \D_l} \varepsilon_N (b_j * \F_{\R^m}^{-1} \Phi_{N,l}(\cdot-(P_1(Nt),\dotsc,P_m(Nt)))(x) \right|  \leqsim_{C_2} \langle l \rangle^{O(1)} \|f\|_{\ell^1(Q_j)} \] where $b_j = b1_{Q_j}$, $b$ is the ``bad part" of $f$, and $\{Q_j\}_{j=1}^\infty$ are a disjoint collection of dyadic blocks whose union is \[ \left\{ x \in \Z^m : \left| \sum_{N \in \D_l} \varepsilon_N (f * \F_{\R^m}^{-1} \Phi_{N,l}(\cdot-(P_1(Nt),\dotsc,P_m(Nt)))(x) \right| \geq \alpha \right\}. \] We can bound \begin{multline*} \sum_{x \in (100 Q_j)^c} \left| \sum_{N \in \D_l} \varepsilon_N (b_j * \F_{\R^m}^{-1} \Phi_{N,l}(\cdot-(P_1(Nt),\dotsc,P_m(Nt)))(x) \right| \\ \leq \sum_{x \in (100Q_j)^c} \sum_{N \in \D_l} \sum_{y \in Q_j} |\F_{\R^m}^{-1} \Phi_{N,l}(x-y-(P_1(Nt),\dotsc,P_m(Nt)) \\-\F_{\R^m}^{-1} \Phi_{N,l}(x-(P_1(Nt),\dotsc,P_m(Nt))| |b_j(y)| \end{multline*} using the fact that $b_j$ has mean zero. By H\"older's inequality, it suffices to show that \begin{multline*} \sum_{x \in (100Q_j)^c} \sum_{N \in \D_l} |\F_{\R^m}^{-1} \Phi_{N,l}(x-y-(P_1(Nt),\dotsc,P_m(Nt))-\F_{\R^m}^{-1} \Phi_{N,l}(x-(P_1(Nt),\dotsc,P_m(Nt))| \\ \leqsim_{C_2} \langle l \rangle^{O(1)} \end{multline*} for all $y \in Q_j$. This argument is standard and can be found in Chapter 1 of \cite{stein1993harmonic}.

Now suppose $l=-C_1$. Then the Fourier transform of \[ \F_{\R^m}^{-1} \Phi_{N,l}(\cdot-(P_1(Nt),\dotsc,P_m(Nt))-\Phi_{N,l} \] vanishes at the origin so the above argument can be used analogously to show that \[ \|V^r(f*(\F_{\R^m}^{-1} \Phi_{N,l}(\cdot-(P_1(Nt),\dotsc,P_m(Nt))-\Phi_{N,l}))_{N \in \D_l} \|_{\ell^p(\Z^m)} \leqsim_{C_2} \langle l \rangle^{O(1)} \|f\|_{\ell^p(\Z^m)}. \] Thus by the triangle inequality in $V^r$ and $\ell^p(\Z^m)$ we have reduced to showing that \[ \|V^r(f*\Phi_{N,l})_{N \in \D_l} \|_{\ell^p(\Z^m)} \leqsim_{C_2} \langle l \rangle^{O(1)} \rangle \|f\|_{\ell^p(\Z^m)}. \] This follows from Theorem 1.1 of \cite{jones2008strong}.
\end{proof}

This completes the proof of Theorem \ref{thm:integervariation}.

\section{Short Variation}

In this section we complete the proof of Theorem \ref{thm:integerfullvariation}. Fix $P_1,\dotsc,P_m$ such that $\{P_,\dotsc,P_m\} \in \P'$. As discussed in the introduction, we first bound \[ V^r(A_Nf)_{N \in \N} \leqsim V^r(A_Nf)_{N \in \D} + \left( \sum_{k=1}^\infty V^2(A_Nf)_{N \in \I_k}^2 \right)^\frac{1}{2} \] where $\D = \{2^n : n \in \N\}$ and $\I_k = \{2^k,\dotsc,2^{k+1}\}$. By the triangle inequality in $\ell^p(\Z^m)$ and Theorem (\ref{thm:integervariation}), using the fact that $\P' \subseteq \P$, we only need to prove \begin{equation} \label{eqn:shortvariationbounded} \left\| \left( \sum_{k=1}^\infty V^2(A_Nf)_{N \in \I_k}^2 \right)^\frac{1}{2} \right\|_{\ell^p(\Z^m)} \leqsim \|f\|_{\ell^p(\Z^m)}. \end{equation}

We fix a partition of unity $\{\eta_j\}_{j=1}^\infty$ so that \[ \sum_{j=1}^\infty \eta_j(\xi)=1 \] for all $\xi \neq 0$ and such that $\eta_j$ is supported on $ \prod_{i=1}^m [[ |P_i(2^j)|^{-1} ]]$ and vanishes on $\prod_{i=1}^m [[ |P_i(2^{j+1})|^{-1} ]]$ for all $i$ and $j$. Thus for all $f \in \ell^p(\Z^m)$ with $p>1$ we have \begin{equation} \label{eqn:partitionofunityequality} f = \sum_{j=1}^\infty f * \check{\eta}_j \end{equation} with convergence in $\ell^p(\Z^m)$. We define $\eta_j=0$ for $j \leq 0$ and let \[ Sf = \left( \sum_{j \in \Z} |f*\check{\eta}_j|^2 \right)^\frac{1}{2} \] denote the square function with respect to this partition of unity. Since the $\eta_j$ are smooth and have lacunary supports, we have \begin{equation} \label{eqn:squarefunction} \|Sf\|_{\ell^p(\Z^m)} \sim \|f\|_{\ell^p(\Z^m)} \end{equation} for all $p>1$. Then by (\ref{eqn:partitionofunityequality}), changing variables in $k$, and Minkowski's inequality in $V^2$, $\ell^2(\N)$, and $\ell^p(\Z^m)$, the left hand side of (\ref{eqn:shortvariationbounded}) is bounded above by \[ \sum_{j \in \Z} \left\| \left( \sum_{k=1}^\infty V^2(A_N(f*\check{\eta}_{j+k}))_{N \in \I_k}^2 \right)^\frac{1}{2} \right\|_{\ell^p(\Z^m)}. \] By interpolation, it therefore suffices to prove \begin{equation} \label{eqn:absjbound} \left\| \left( \sum_{k=1}^\infty V^2(A_N(f*\check{\eta}_{j+k}))_{N \in \I_k}^2 \right)^\frac{1}{2} \right\|_{\ell^p(\Z^m)} \leqsim 2^{-O(|j|)1_{p=2}} \|f\|_{\ell^p(\Z^m)} \end{equation} for each $j \in \Z$.

We begin with the $p=2$ case. Then \[ \left\| \left( \sum_{k=1}^\infty V^2(A_N(f*\check{\eta}_{j+k}))_{N \in \I_k}^2 \right)^\frac{1}{2} \right\|_{\ell^2(\Z^m)} = \left( \sum_{k=1}^\infty \|V^2(A_N(f*\check{\eta}_{j+k}))_{N \in \I_k} \|_{\ell^2(\Z^m)}^2 \right)^\frac{1}{2}. \] We claim that \begin{equation} \label{eqn:shortell2bound} \|V^2(A_N(f*\check{\eta}_{j+k}))_{N \in \I_k} \|_{\ell^2(\Z^m)} \leqsim 2^{-O(|j|)} \|f*\check{\eta}_{j+k}\|_{\ell^2(\Z^m)}. \end{equation} To prove (\ref{eqn:shortell2bound}) we will use the following argument which was used by Bourgain in \cite{Bourgain89} (Lemma 3.11), see also Lemma 3.20 of \cite{krausediscrete}.

\begin{lemma} \label{lem:general2shortbound}
Fix $k \in \N$, $f \in \ell^2(\Z^m)$, and smooth functions $\phi_1,\dotsc,\phi_{2^k}$ on $\T^m$. Suppose that \[ \sup_{\xi \in \mathrm{supp} \hat{f}} |\phi_n(\xi)| \leq A \] and \[ \sup_{\xi \in \mathrm{supp} \hat{f}} |\phi_{n+1}(\xi)-\phi_n(\xi)| \leq a \] for all $n$. Then $\|V^2(f*\check{\phi}_N)_{N \in [2^k]} \|_{\ell^2(\Z^m)} \leqsim \sqrt{2^kAa} \|f\|_{\ell^2(\Z^m)}$.
\end{lemma}

To use Lemma \ref{lem:general2shortbound}, note that since variation operators are defined using differences, we are free to replace $A_N$ by $B_N = A_N-A_{2^k}$. We denote $\nu_N = m_{N;\Z}-m_{2^k;\Z}$. By the triangle inequality we have \begin{equation} \|\nu_{N+1}-\nu_N\|_{L^\infty(\T)} = \|m_{N+1;\Z}-m_{N;\Z}\| \leqsim N^{-1} \leqsim 2^{-k} \end{equation} for all $N \in \I_k$. We now bound $|\nu_N(\xi)|$. On the one hand, by the triangle inequality and the Mean Value theorem we have \[ |\nu_N(\xi)| \leq |m_{N;\Z}(\xi)-1| + |m_{2^k}(\xi)-1| \leqsim \sum_{i=1}^m |\xi_i||P_i(2^k)| \leqsim \max_{1 \leq i \leq m} |\xi_i||P_i(2^k)| \] for all $N \in \I_k$. On the other hand, by the triangle inequality and (\ref{eqn:directform}) we have \[ |\nu_N(\xi)| \leqsim 2^{-Ak} + \left(\max_{1 \leq i \leq m} |\xi_i| |P_i(2^k)| \right)^{-B} \] for all $N \in \I_k$, where $A,B>0$ are the implicit constants written explicitly. On choosing $B>0$ sufficiently small but depending only on $A,P_1,\dotsc,P_m$, we can ensure that \[ 2^{-Ak} \leq \left(\max_{1 \leq i \leq m} |\xi_i| |P_i(2^k)| \right)^{-B} \] for all $\xi \in \T^m$ and $N \in \I_k$. Thus \begin{equation} \label{eqn:nuNbound} |\nu_N(\xi)| \leqsim \min \left( \max_{1 \leq i \leq m} |\xi_i| |P_i(2^k)| , \left(\max_{1 \leq i \leq m} |\xi_i| |P_i(2^k)| \right)^{-O(1)} \right). \end{equation}

We can now use Lemma \ref{lem:general2shortbound} to prove (\ref{eqn:shortell2bound}). Since $|\I_k| \leqsim 2^{-k}$, it suffices to prove that \[ \sup_{\xi \in \mathrm{supp} \eta_{j+k}} |\nu_N(\xi)| \leqsim 2^{-O(|j|)} \] for all $N \in \I_k$. This is trivial if $j+k \leq 0$. If $j \geq 0$ then by the first bound in (\ref{eqn:nuNbound}) and (\ref{eqn:noepsHardybound}) we have \[ \sup_{\xi \in \mathrm{supp} \eta_{j+k}} |\nu_N(\xi)| \leqsim \max_{1 \leq i \leq m} |P_i(2^{j+k})|^{-1}|P_i(2^k)| \leqsim 2^{-O(j)} = 2^{-O(|j|)}. \] Now suppose $-k<j<0$. Then by the second bound of (\ref{eqn:nuNbound}) and (\ref{eqn:noepsHardybound}) again we have \[ \sup_{\xi \in \mathrm{supp} \eta_{j+k}} |\nu_N(\xi)| \leqsim \left( \max_{1 \leq i \leq m} |P_i(2^{j+k})|^{-1} |P_i(2^k)| \right)^{-O(1)} \leqsim 2^{O(j)} = 2^{-O(|j|)}. \] This proves (\ref{eqn:shortell2bound}).

To complete the proof of (\ref{eqn:absjbound}) in the $p=2$ case, we bound \[ \left( \sum_{k=1}^\infty \|V^2(A_N(f*\check{\eta}_{j+k}))_{N \in \I_k} \|_{\ell^2(\Z^m)}^2 \right)^\frac{1}{2} \leqsim 2^{-O(|j|)} \left( \sum_{k=1}^\infty \|f * \check{\eta}_{j+k} \|_{\ell^2(\Z^m)}^2 \right)^\frac{1}{2} \] and \[ \left( \sum_{k=1}^\infty \|f * \check{\eta}_{j+k} \|_{\ell^2(\Z^m)}^2 \right)^\frac{1}{2} = \left\| \left( \sum_{k=1}^\infty |f*\check{\eta}_{j+k}|^2 \right)^\frac{1}{2} \right\|_{\ell^2(\Z^m)} \leq \|Sf\|_{\ell^2(\Z^m)} \leqsim \|f\|_{\ell^2(\Z^m)} \] by (\ref{eqn:squarefunction}).

Now suppose $p \neq 2$. Note that \begin{equation} V^2(A_Nf)_{N \in \I_k} \leqsim \sum_{N \in \I_k} |A_{N+1}f-A_Nf| \leqsim A_{2^{k+1}}|f| \end{equation} by monotonicity of $V^r$ norms and the triangle inequality. We will use the following lemma which can be found in \cite{krausediscrete} (Lemma 3.16).

\begin{lemma} \label{lem:vectorinequality}
Given a collection $\{f_k\}_{k=1}^\infty$ of functions on $\Z$, we have \[ \left\| \left( \sum_{k=1}^\infty | A_{2^k} f_k|^2 \right)^\frac{1}{2} \right\|_{\ell^p(\Z^m)} \leqsim \left\| \left( \sum_{k=1}^\infty |f_k|^2 \right)^\frac{1}{2} \right\|_{\ell^p(\Z^m)} \] for any $p>1$.
\end{lemma}

\begin{proof}
We observe that \[ \E_{n \in [N]} f(x+(\lfloor P_1(n) \rfloor,\dotsc,\lfloor P_m(n) \rfloor)) = A_Ng(-x) \] where $g(x)=f(-x)$. Thus by the maximal inequality we have already obtained (by Theorem \ref{thm:integervariation} and Corollary \ref{cor:maximal}), we have \[ \left\| \sup_{k \in \N} |\E_{n \in [2^k]} f(\cdot+(\lfloor P_1(n) \rfloor,\dotsc,\lfloor P_m(n) \rfloor))| \right\|_{\ell^p(\Z^m)} = \|\sup_{k \in \N} |A_{2^k}g| \|_{\ell^p(\Z^m)} \leqsim \|g\|_{\ell^p(\Z^m)} = \|f\|_{\ell^p(\Z^m)} \] for any $p>1$. The result follows by Lemma 3.16 of \cite{krausediscrete}.
\end{proof}

To handle (\ref{eqn:absjbound}) we bound \[ \left\| \left( \sum_{k=1}^\infty V^2(A_N(f*\check{\eta}_{j+k}))_{N \in \I_k}^2 \right)^\frac{1}{2} \right\|_{\ell^p(\Z^m)} \leqsim \left \| \left( \sum_{k=1}^\infty (A_{2^k}|f*\check{\eta}_{j+k}|)^2 \right)^\frac{1}{2} \right\|_{\ell^p(\Z^m)}. \] By Lemma \ref{lem:vectorinequality} this is bounded above by \[ \left\| \left( \sum_{k=1}^\infty |f*\check{\eta}_{j+k}|^2 \right)^\frac{1}{2} \right\|_{\ell^p(\Z^m)} \leq \|Sf\|_{\ell^p(\Z^m)} \leqsim \|f\|_{\ell^p(\Z^m)}. \] This concludes the proof of (\ref{eqn:absjbound}) and hence of (\ref{eqn:shortvariationbounded}). Theorem \ref{thm:integerfullvariation} follows.

\printbibliography
\end{document}